\documentclass{cimart}

\newcommand{\Aref}[1]{\hyperref[A#1]{(A#1)}}
\newcommand{\Sref}[1]{\hyperref[S#1]{(S#1)}}

\allowdisplaybreaks

\title[On the behavior of free boundaries to generalized two-phase Stefan problems]{On the behavior of free boundaries to generalized two-phase Stefan problems for parabolic partial differential\\equation systems}

\authors{Toyohiko Aiki and Hana Kakiuchi}

\authorinfo[T. Aiki]{Japan Women's University, Japan}{    aikit@fc.jwu.ac.jp}

\authorinfo[H. Kakiuchi]{Japan Women's University, Japan}{m1816022kh@ug.jwu.ac.jp}

\abstract{%
   Recently, we have proposed a new free boundary problem representing the bread baking process in a hot oven. Unknown functions in this problem are the position of the evaporation front, the temperature field, and the water content.  In solving this problem, we observed two difficulties: the growth rate of the free boundary depends on the water content, and the boundary condition for the water content involves the temperature. In this paper, by improving the regularity of solutions, we overcome these difficulties and establish the existence of a solution locally in time and its uniqueness. Moreover, under some sign conditions for initial data, we derive a result on the maximal interval of existence for solutions.}



\VOLUME{34}
\YEAR{2026}
\ISSUE{1}
\NUMBER{6}
\DOI{https://doi.org/10.46298/cm.17254}
\licence{Non-exclusive license to distribute}
\editinfo{January 9, 2026}{February 17, 2026}{Serena Dipierro}

\begin{document}
\section{Introduction}
\label{sec1}
On the bread baking process, several mathematical results have been investigated in \cite{MD, ZP, ZPP}. In particular, in \cite{PS}, the free boundary problem was proposed as a mathematical model for the process and considered in the enthalpy formulation. The enthalpy formulation allows the treatment of the multidimensional domains. However, it does not describe the behavior of the free boundary directly. Therefore, in our previous result  \cite{AK}, we derived the following one-dimensional free boundary problem and established existence and uniqueness theorems only for its approximation problems, due to certain mathematical difficulties that will be discussed later in this section. The aim of the present paper is to show the existence, uniqueness, and the behavior of the solutions to the free boundary problem under suitable assumptions. 
The problem is to find a triplet $(e, u, w)$ of the evaporation front $x=e(t) \text{ for } 0 < t < T$, the temperature field $u = u(t,x)$ and the water content $w=w(t,x)$ for $(t,x) \in Q(T) := (0,T) \times (0,1)$, where $T > 0$ is a given time, $t$ is the time variable and $x$ indicates the position. Here, we note that $u$ is given by $u = \theta - \theta_c$, where $\theta$ is the temperature and $\theta_c$ is a positive constant indicating the phase transition from water to air. In this model we assume that the bread occupies the one-dimensional interval $(0,1)$, and the common domain of $u$ and $w$ consists of the crumb region \[ Q_l(T, e) = \{ (t,x) \mid 0 < x < e(t) \text{ for } 0 < t < T\}\] and the crust region
\[
  Q_a(T, e) = \{ (t,x) \mid e(t) < x < 1 \text{ for } 0 < t < T \}.
\]
(see \cite{AK} for details). Also, the triplet $(e, u, w)$ satisfies:
 \begin{align}
 c_l \frac{\partial u}{\partial t} &= k_l \frac{\partial^2 u}{\partial x^2} \text{ in } Q_l(T, e), \quad c_a \frac{\partial u}{\partial t} = k_a \frac{\partial^2 u}{\partial x^2} \text{ in } Q_a(T, e),  \label{EQl}  \\
 \frac{\partial w}{\partial t} &= d_l \frac{\partial^2 w}{\partial x^2} \text{ in } Q_l(T, e),  \quad \frac{\partial w}{\partial t} = d_a \frac{\partial^2 w}{\partial x^2} \text{ in } Q_a(T, e), \label{EQa} \\
\frac{\partial u}{\partial x}(t,0) &= 0, \quad  u(t,e(t)) = 0 \text{ for } 0 < t < T,  \label{BCl} \\
 - k_a \frac{\partial u}{\partial x}(t,1) &= h( u(t,1) + \theta_c - u_b(t)) + \sigma( (u(t,1) + \theta_c)^4 - u_b(t)^4) \text{ for } 0 < t < T, \label{BCau} \\
 - d_a \frac{\partial w}{\partial x}(t,1) &= b_1p(u(t,1)+\theta) -  b_2 p(u_b(t)) \text{ for } 0 < t < T,  \label{BCaw} \\
 l w(t,e(t))  e'(t) &= k_l \frac{\partial u}{\partial x}(t, e(t)-) -   k_a \frac{\partial u}{\partial x}(t, e(t)+) \text{ for } 0 < t < T,  \label{FBP1} \\
   \frac{\partial w}{\partial x}(t, 0) &= 0,  \quad  d_l \frac{\partial w}{\partial x}(t, e(t)-) = d_a \frac{\partial w}{\partial x}(t, e(t)+) \text{ for } 0 < t < T,  \label{FBP2} \\
e(0) &= e_0, u(0) = u_0, w(0) = w_0 \text{ on } [0,1],\label{IC}
\end{align} 
where $c_l$ and $c_a$ are the specific heats, $k_l$ and $k_a$ are the thermal conductivities, and $d_l$ and $d_a$ are the water diffusion coefficients in the crumb and the crust regions, respectively. In addition, $h$ and $\sigma$ are non-negative constants corresponding to the heat transfer constant at the boundary at $x=1$ and the Stefan--Boltzman constant, respectively, $u_b$ is a given function on $[0,T]$ and indicates the temperature of the oven, $b_1$ and $b_2$ are positive constants and $p$ is a continuous function on $\mathbb{R}$. Moreover, $l$ is the latent heat, $e_0$ is the initial position of the free boundary, and $u_0$ and $w_0$ are the initial temperature field and the water content, respectively. Throughout this paper P denotes the system \eqref{EQl}-\eqref{IC}.
\par We refer  \cite{AK} for the physical background in detail. Here, we explain the derivation of the system \eqref{EQl}-\eqref{IC}, briefly. Equation \eqref{EQl} is the standard heat equation and equation \eqref{EQa} describes the conservation law of the mass of water. The homogeneous Neumann boundary conditions \eqref{BCl} and \eqref{FBP2} at $x=0$ follow from the symmetry of the bread, and the boundary conditions \eqref{BCau} and \eqref{BCaw} at $x=1$ are assumed in \cite{PS}. In this paper we call the equation \eqref{FBP1} the generalized Stefan condition, since the water content $w$ appears as the coefficient of the time derivative of the free boundary. Moreover, \eqref{FBP2} is called the transmission condition.
\par First, we show features of the system P as follows. As mentioned above, the generalized Stefan condition contains the water content on the free boundary in the coefficient of the growth rate. Accordingly, in order to discuss the solvability we need an estimate for the minimum value of the water content $w$. Also, the boundary value $u(\cdot,1)$ appears in the boundary condition for $w$. For existence of  strong solutions for the two-phase Stefan problem we suppose that $u_0 \in W^{1,2}(0,1)$ as usual. However, in this case the regularity $u_t(\cdot,1) \in L^2(0,T)$ is not guaranteed, and then  existence of a strong solution $w$ is not proved. To address these difficulties, in \cite{AK} we approximate the  boundary condition \eqref{BCaw} by applying the mollifier to $u(\cdot,1)$ and establish existence and uniqueness based on the fixed-point argument.
\par The first main result of this paper is to obtain strong solutions of P under the high regularity $u_0 \in W^{2,2}(0,1)$. In the proof, the uniform estimate for time difference is essential (see Lemma \ref{lem4}). Our proof is due to \cite{Ke1, Ai,AK}.
\par Moreover, we show a result on the behavior on the free boundary. To present the result, briefly, let $T^*$ be a maximal time for the existence of the solution. In our problem P one of the following conditions holds.
\begin{enumerate}
\item[(a)] $T^* = \infty$
\item[(b)] $T^* < \infty$, $e(t) \to 0$ as $t \uparrow T^*$
\item[(c)] $T^* < \infty$, $e(t) \to 1$ as $t \uparrow T^*$
\end{enumerate}
This type of behavior result was already obtained by \cite{MYY} and  \cite{Ke3}. To show the behavior of the free boundary we assume that the temperature $u_b$ in the oven is constant and the initial value $w_0$ is strictly positive on the domain $(0,1)$.
\section{Main Results}
\label{sec2}
First in this section, we define a solution to P  and a theorem concerned with the uniqueness and the existence of the solution.
\begin{definition}\label{def1}
For given time $T$ a triplet $(e, u, w)$ of functions $e$, $u$ and $w$ is a solution of P on $[0,T_0]$ for  $0 < T_0 \leq T$ if $(e, u, w)$  satisfies the following conditions:
\begin{enumerate}
\item[(S1)]\phantomsection\label{S1} $e \in W^{1, \infty}(0, T_0)$ and $0 < e < 1  \text{ on }  [0, T_0]$.
\item[(S2)]\phantomsection\label{S2} $u \in W^{1,2}(0,T_0;L^2(0, 1)) \cap L^{\infty} (0, T_0; W^{1,2}(0, 1))$, $u_{xx}  \in L^2(Q_l(T_0, e))$, $L^2(Q_a(T_0, e))$, and $u_x(\cdot, e(\cdot)\pm) \in L^{\infty}(0,T_0)$.
\item[(S3)]\phantomsection\label{S3} $u(\cdot, 1) \in W^{1,2}(0,T_0)$.
\item[(S4)]\phantomsection\label{S4} $w \in W^{1,2}(0,T_0;L^2(0, 1)) \cap L^{\infty} (0, T_0; W^{1,2}(0, 1))$, $w_{xx}  \in L^2(Q_l(T_0, e))$, $L^2(Q_a(T_0, e))$, $w(t, e(t)) \geq \delta_1 \text{ for any } t \in [0,T_0]$, where $\delta_1$ is some positive constant.
\item[(S5)]\phantomsection\label{S5} \eqref{EQl}-\eqref{IC} hold.
\end{enumerate}
\end{definition}
Before showing Theorem \ref{th1}, we state the assumptions.
\begin{enumerate}
\item[(A1)]\phantomsection\label{A1} $p \in C^1({\mathbb R})$, $0 \leq p, p' \leq M_p$ on ${\mathbb R}$,  where $M_p$ is a positive constant.
\item[(A2)]\phantomsection\label{A2} $0 <  e_0 < 1$, $u_0 \in W^{1,\infty}(0, 1)$,  $u_0 \geq 0$ on $[e_0,1]$, $u_0 \leq 0$ on  $[0,e_0]$.
\item[(A3)]\phantomsection\label{A3} $w_0  \in W^{1,2}(0,1)$, $w_0(e_0) > 0$.
\item[(A4)]\phantomsection\label{A4} $u_b \in W^{1,2}(0,T)$, $u_b \geq \theta_c \text{ on } [0,T]$.
\item[(A5)]\phantomsection\label{A5} $u_0 \in W^{2,2}(e_0,1)$, $u_{0x}(0) = 0$,\\
$-k_a u_{0x}(1) =  h( u_{0 }(1) + \theta - u_b(0)) + \sigma( (u_{0 }(1) + \theta)^4 - u_b(0)^4)$.
\end{enumerate}
\begin{theorem}\label{th1}
 Assume \Aref{1}--\Aref{5}. Then  P has at least one solution $(e, u, w)$ on $[0,T_0]$ for some $T_0 \in (0,T]$. Also, for any $0 < T_0 \leq T$, P admits at most one solution on $[0,T_0]$  in the sense of Definition \ref{def1}.
\end{theorem} 
\begin{remark}\label{re1-1}
 The physical relevance of \Aref{1}--\Aref{5} is as follows. The function $p = p(\theta)$ indicates the water vapor pressure at temperature $\theta$, namely, its specific form is
\begin{equation*}
p(\theta) = k_p \exp \left(A_p - \frac{B_p}{\theta + C_p}\right) \text{ for } \theta>0,
\end{equation*}
where $k_p$, $A_p$, $B_p$ and $C_p$ are positive constants. Accordingly, \Aref{1} is physically valid. In this paper, we assume that there exist crumb, crust and evaporation front at the initial. Under this assumption, \Aref{2} holds. Since we consider that every crumbs and crusts contain a non-zero amount of water, \Aref{3} is satisfied. On \Aref{4}, when we bake bread, the oven temperature is about $200^{\circ}$C, that is, the temperature $u_b$ oh the hot air is greater than $\theta_c$. Finally, \Aref{5} is imposed for mathematical reasons.

\end{remark}
\begin{remark}\label{re1}
 We already established existence and uniqueness of solutions to approximate problems under \Aref{1}--\Aref{4} in \cite{AK}. In this paper by assuming \Aref{5}, we succeed to obtain uniform estimates for the time difference $\frac{u(t) - u(t - \Delta t)}{\Delta t}$ with $\Delta t > 0$, which leads to the regularity $u(\cdot,1) \in W^{1,2}(0,T)$. This regularity allows us to prove existence of a function $w$ satisfying the water diffusion system. However, due to dependence on the free boundary, it is not easy to show estimates for the difference of solutions to the water diffusion system. In order to overcome this difficulty, we employ the $L^{\infty}$-estimate for $u_x$ given in \cite{LSU}. Furthermore, the application of the Gagliardo-Nirenberg inequality is also the key of our proof to Theorem \ref{th1} (see Lemma \ref{lemma101})
\end{remark}
\begin{definition}\label{def-max}
 For $T^* \leq \infty$, we call that $[0,T^*)$ is the maximal interval of existence of the solution of P, if there exists a solution on $[0,T]$ for any $0 < T < T^*$, and do not exist on $ [0,T^*]$ for $T^* < \infty$. For $T^* = \infty$, P has a solution on $[0,T]$ for any $T>0$.
\end{definition}
\begin{theorem}\label{th2}
 Assume \Aref{1}--\Aref{3} and \Aref{5}. If $b_1 \leq b_2$, $u_b$ is a positive constant such that $u_b \geq u_0 + \theta_c$, and $w_0 > 0$ on $[0,1]$. For $0 < T^* \leq \infty$, let $[0,T^*)$ be the maximal interval of existence of the solution $(u, w, e)$ to P. Then, one and only one of the following cases (a), (b) and (c) always hold:
\begin{enumerate}
\item[(a)] $T^* = \infty$.
\item[(b)] $T^* < \infty$ and $\displaystyle \lim_{t \uparrow T^*}e(t) = 0$.
\item[(c)] $T^* < \infty$ and $\displaystyle \lim_{t \uparrow T^*}e(t) = 1$.
\end{enumerate}
\end{theorem}
In our proofs of Theorems \ref{th1} and \ref{th2}, Proposition \ref{prop1} which guarantees the integrability of the time derivative for the boundary function plays a very important role. 
\begin{proposition}\label{prop1}
 Under the same assumption as in Theorem \ref{th1}, let $(e,u,w)$ be a solution of P on $[0,T]$. Then, $u(\cdot, 1) \in W^{1,2}(0,T)$. Moreover, if $e \leq 1-\delta$ on $[0,T]$ for some $\delta > 0$, then there exists a positive constant $R_b$ depending only on $\delta > 0$, $\left|u_b\right|_{W^{1,2}(0,T)}$, $\left|u_{0xx}\right|_{L^{2}(1-\delta,1)}$ and $\left|u_0\right|_{W^{1,2}(0,1)}$ such that 
\begin{equation*}
\left|u(\cdot,1)\right|_{W^{1,2}(0,T)} \leq R_b \text{ and } \left|u_t(t)\right|_{L^2(1-\delta / 2,1)} \leq R_b \text{ for } 0 \leq t \leq T.
\end{equation*}
\end{proposition}
We shall prove this proposition in Section \ref{sec3}.
\rm \section{\text{Known results}} 
\label{sec10}
First, we give a list of useful inequalities.
\begin{lemma}\label{lemgag}
 For $a_1 < a_2$, there exists a positive constant $C_0 = C_0(a_1,a_2)$ such that
\begin{align*}
&\left|z\right| \leq C_0( \left|z_x\right|_{L^2(a_1,a_2)}^{1/2}\left|z\right|_{L^2(a_1,a_2)}^{1/2} + \left|z\right|_{L^2(a_1, a_2)}) \text{ on } [a_1, a_2]  \text{ for } z \in W^{1,2}(a_1,a_2),\\
&\left|z(a_2)\right| \leq C_0 \left|z_x\right|_{L^2(a_1,a_2)}^{1/2}\left|z\right|_{L^2(a_1,a_2)}^{1/2} \text{ for } z \in W^{1,2}(a_1,a_2) \text{ with } z(a_1)=0.
\end{align*}
\end{lemma}
Next, we consider auxiliary problems
\begin{gather*}
    \text{AP}(e)
    := \{ \eqref{EQl}-\eqref{BCaw}, \eqref{FBP2}, \eqref{IC} \}, 
    \\
    \text{AP1}(e, u_0)
    := \{ \eqref{EQl}, \eqref{BCl}, \eqref{BCau}, \eqref{mazda} \},
    \quad \text{and} \quad
    \text{AP2}(e, b)
    := \{ \eqref{EQa}, \eqref{honda}, \eqref{toyota} \}
\end{gather*}
for given $e \in C([0,T])$ with $0 < e < 1$ on $[0,T]$ and $b \in C([0,T])$ as a preliminary step. Here, \eqref{mazda}, \eqref{honda} and \eqref{toyota} are as follows:
\begin{align}
u(0) &= u_0 \text{ on } [0,1], \label{mazda}\\
\frac{\partial w}{\partial x}(t,0) &= 0, \quad - d_a \frac{\partial w}{\partial x}(t,1) = b_1p(b(t)) - b_2 p(u_b(t)) \text{ for } 0 < t < T, \label{toyota}\\
w(0) &= w_0 \text{ on } [0,1]. \label{honda}
\end{align}
For the problems AP1($e, u_0$) and AP2($e,b$) we have already obtained the following lemmas.
\begin{lemma}\label{prop3-1}
(cf. \cite[Proposition 3.1 and Lemma 3.8]{AK}) Let $T>0$ and $e \in W^{1,2}(0,T)$ satisfying $0<e<1 \text{ on } [0,T]$. If \Aref{2} and \Aref{4} hold, then AP1($e, u_0$) has one and only one solution $u$ satisfying \Sref{2}. Moreover, $u \leq 0$ in $Q_l(T,e)$ and $u \geq 0$ in $Q_a(T,e)$.
\end{lemma}
\begin{lemma}\label{prop3-2}
 Let $T>0$, $e \in W^{1,2}(0,T)$ with $0<e<1 \text{ on } [0,T]$ and $w_0 \in W^{1,2}(0,1)$. If $b \in W^{1,2}(0,T)$, then AP2($e, b$) has a unique solution
\[
w \in W^{1,2}(0,T;L^2(0, 1)) \cap L^{\infty} (0, T; W^{1,2}(0,1)),
\quad
w_{xx} \in L^2(Q_l(T, e)),\; L^2(Q_a(T, e)).
\]
\end{lemma}
We can prove Lemma \ref{prop3-2} in the similar way to that of \cite[Proposition 3.2]{AK}, on account of $b \in  W^{1,2}(0,T)$. Thus, we omit the proof of the lemma.
\par Due to the results obtained in \cite{AK}, we provide some uniform estimates for solutions of AP1($e, u_0$). Here, we put
\begin{align*}
K(\delta, e_0, M,T_0) =\Big \{ e \in W^{1,3}(0,T) \mid e(0) = e_0, \delta \leq e \leq 1 - \delta \text{ on } [0, T_0], \int_0^{T_0} \left|e'\right|^3 dt \leq M
 \Big \},
\end{align*}
where $\delta \in (0,1)$, $e_0 \in [\delta, 1-\delta]$, $M>0$ and $0 < T_0 \leq T$.
\begin{lemma}
 [cf. Lemmas 4.2, 4.4 and 4.5 in \cite{AK}]\label{lemma100} Let $\delta > 0$, $e_0 \in [\delta, 1-\delta]$, $M>0$, $T > 0$ and $e \in K(\delta, e_0, M,T)$ and assume \Aref{2} and \Aref{4}. If $u$ is a solution of AP1($e, u_0$), then for some $C_1 = C_1(\delta, M,T,\left|u_0\right|_{W^{1,2}(0,1)}) > 0$ such that
\begin{align}
\int^1_0 \left|u(t,x)\right|^2 dx +  \int^{t}_{0} \!\!\! \int^{1}_{0} \left| u_x(\tau,x)\right|^2 dx d\tau  \leq C_{1} \text{ for } 0 \leq t \leq T,  \label{tana}
\end{align}
and
\begin{align}
\int^1_0 \left|u_x(t,x)\right|^2 dx +  \int^{t}_{0} \!\!\! \int^{1}_{0} \left| u_{\tau}(\tau,x)\right|^2 dx d\tau  \leq C_{1} \text{ for } 0 \leq t \leq T. \label{kiku}
\end{align}
Moreover, there exists a positive constant $C_2 = C_2(\delta, M,T,\left|u_0\right|_{W^{1,\infty}(0,1)}) $ such that
\begin{align}
\left|u_x(t, e(t)-)\right|, \left|u_x(t, e(t)+)\right|  \leq C_{2} \text{ for }  0 \leq t \leq T. \label{maru}
\end{align}
\end{lemma}
\begin{proof}
In \cite{AK}, by putting
\begin{align*}
\overline{u}(\cdot,y) = 
\begin{cases}
u(\cdot,y e) \text{ for }  0 \leq y \leq 1, \\
u(\cdot,e+(y-1)(1-e)) \text{ for }   1 \leq y \leq 2 \text{ on } [0,T],
\end{cases}
\end{align*}
we obtained the existence of a positive constant  $R_1 = R_1\left(\delta, M,T,\left|u_0\right|_{W^{1,2}(0,1)}\right)$ such that
\begin{align}
&\frac{c_*}{2} \int^2_0 \left|\overline{u}(t,y)\right|^2 dy + \frac{k_a}{2} \int^{t}_{0} \!\!\! \int^{2}_{0} \left|\overline{u}_y(\tau,y)\right|^2 dy d\tau  \leq R_1 \text{ for } 0 \leq t \leq T,\label{hoshi1}\\
& \frac{k_l}{2e(t)} \int^{1}_{0} \left|\overline{u}_y(t,y)\right|^2 dy + \frac{k_a}{2(1-e(t))}  \int^{2}_{1} \left| \overline{u}_y(t,y)\right|^2 dy \notag\\
&\quad \, +\frac{\delta c_a}{2} \int^{t}_{0} \!\!\! \int^{2}_{0} \left| \overline{u}_{\tau}(\tau,y) \right|^2 dy d\tau \leq R_1 \text{ for } 0 \leq t \leq T,\label{hoshi2}
\end{align}
where $c_* = \min\{c_l, c_a\}$. It is clear that $\overline{u}_y(t,y) = e(t) u_x(t,y e(t))  \text{ for }  0 \leq y \leq 1$ and
\[
\overline{u}_y(t,y) = (1-e(t))u_x(t,e(t)+(y-1)(1-e(t))) \text{ for }   1 \leq y \leq 2,
\]
and $0 \leq t \leq T$. Thus, we get \eqref{tana} and \eqref{kiku}. Moreover, \eqref{maru} is a direct consequence of Lemma 4.5 in \cite{AK}.
\end{proof}
In the rest of this section we recall some results covered with estimates for the difference of solutions to AP1($e, u_0$).
\begin{lemma}\label{lemma101}
 Let $\delta > 0$, $M>0$, $T > 0$, $e_{0i} \in [\delta, 1-\delta]$ and $e_i \in K(\delta, e_{0i}, M,T)$ for $i=1,2$ and assume \Aref{2}. If $u_i$ is a solution of AP1($e_i, u_0$) for $i =1,2$, then, for some positive constants $C_3 = C_3(\delta, M,T, \left|u_0\right|_{W^{1,2}(0,1)})$ it holds that
\begin{samepage}
\begin{align}
&\int^1_0 \left|u_1(t) - u_2(t)\right|^2 dx +  \int^{t}_{0} \!\!\! \int^{1}_{0} \left| u_{1x}(\tau) - u_{2x}(\tau) \right|^2 dx d\tau \notag \\
& \leq C_{3} \int^{t}_0 (\left|e_{1} - e_{2}\right|^2 + \left|e'_{1} - e'_{2}\right|^2) d\tau \text{ for } 0 \leq t \leq T.  \label{gen}
\end{align}
\end{samepage}
Moreover, there exists a positive constant $C_4 = C_4(\delta, M,T,  \left|u_0\right|_{W^{1,\infty}(0,1)}) $ such that
\begin{align}
&\int^t_0 \left|u_{1x}(\tau,e_1(\tau)-) - u_{2x}(\tau,e_2(\tau)-)\right|^3 d\tau\notag\\
&\quad \,  +  \int^t_0 \left|u_{1x}(\tau,e_1(\tau)+) - u_{2x}(\tau ,e_2(\tau)+)\right|^3 d\tau\notag\\
& \leq C_{4} t^{1/4} \int^{t}_0 (\left|e_{1} - e_{2}\right|^3 + \left|e'_{1} - e'_{2}\right|^3) d\tau \text{ for } 0 \leq t \leq T. \label{tono}
\end{align}
\end{lemma}
\begin{proof}
Let
\begin{align*}
\overline{u}_i(\cdot,y) = 
\begin{cases}
u_i(\cdot,y e_i) \text{ if } 0 \leq y \leq 1, \\
u_i(\cdot,e_i+(y-1)(1-e_i)) \text{ if } 1 \leq y \leq 2, \text{ on } [0,T] \text{ for } i = 1, 2.
\end{cases}
\end{align*}
Then, Lemma 5.1 in \cite{AK} guarantees existence of a positive constant $R_2 = R_2(\delta, M, T)$ such that
\begin{align}
& \frac{c_l}{2} \int^1_{0} \left|\overline{u}_{1}(t,y) - \overline{u}_{2}(t,y)\right|^2 dy + \frac{k_l}{2} \int^{t}_0 \!\!\! \int^1_0 \left( \overline{u}_{1y}(\tau,y)-  \overline{u}_{2y}(\tau,y)\right)^2 dy d\tau\notag\\
&  \leq R_2 \int^{t}_0 (\left|e_{1}(\tau) - e_{2}(\tau)\right|^2 + \left|e'_{1}(\tau) - e'_{2}(\tau)\right|^2) d\tau,\label{vr2}\\ 
&\frac{c_a}{2} \int^2_{1} \left|\overline{u}_{1}(t,y) - \overline{u}_{2}(t,y)\right|^2 dy + \frac{k_a}{2}   \int^{t}_0 \!\!\! \int^2_1 \left( \overline{u}_{1y}(\tau,y)-  \overline{u}_{2y}(\tau,y)\right)^2 dy d\tau\notag\\
&  \leq R_2 \int^{t}_0 (\left|e_{1}(\tau) - e_{2}(\tau)\right|^2 + \left|e'_{1}(\tau) - e'_{2}(\tau)\right|^2) d\tau \text{ for } 0 \leq t \leq T.\notag
\end{align} 
Hence, we can obtain \eqref{gen}, easily.
\par Next, in order to prove \eqref{tono} we put  $q(y) = -(y-1)^2 + 1$ for $y \in [0,2]$ and $ \overline{v}_{i } =q \overline{u}_{i } $ on $Q(T)$ for $i=1, 2$. Due to Lemmas 5.3 and 5.4 in \cite{AK}, there exists a positive constant $R_3 = R_3(\delta, M, T)$ such that
 \begin{align}
 &\frac{k_l}{2} \int^{1}_{0} \left|\overline{v}_{1 y}(t) - \overline{v}_{2 y}(t)\right|^2 dy
 + \frac{c_l \delta^2}{4} \int^{t}_{0} \!\!\! \int^{1}_{0} \left|\overline{v}_{1t} - \overline{v}_{2t}\right|^2 dyd\tau \notag\\
 &\quad \leq R_3 \int^{t}_0 \left(\left|e_{1 } - e_{2 }\right|^2 + \left|e_{1 }' - e_{2 }'\right|^2\right) d\tau, \label{vr3}\\
 &\frac{k_a}{2} \int^{2}_{1} \left|\overline{v}_{1 y}(t) - \overline{v}_{2 y}(t)\right|^2 dy
 + \frac{c_a \delta^2}{4} \int^{t}_{0} \!\!\! \int^{2}_{1} \left|\overline{v}_{1t} - \overline{v}_{2t}\right|^2 dyd\tau \notag\\
 &\quad \leq R_3 \int^{t}_0 \left(\left|e_{1 } - e_{2 }\right|^2 + \left|e_{1 }' - e_{2 }'\right|^2\right) d\tau \notag\\
 &\text{for } 0 \leq t \leq T.\notag
   \end{align}
    Here, we note that
 \begin{align}
&    \displaystyle c_l \frac{\partial \overline{v}_{i }}{\partial t} = \frac{k_l}{e_{i }^2} \left(\frac{\partial^2 \overline{v}_{i }}{\partial y^2} - 2\frac{\partial \overline{u}_{i }}{\partial y}q' - \overline{u}_{i} q''\right) + \frac{c_lye'_{i }}{e_{i }}\left(\frac{\partial \overline{v}_{i}}{\partial y} - \overline{u}_{i } q'\right) \text{ in } (0,T) \times (0,1),\label{sharp}
  \end{align}
$\displaystyle \max_{0\leq y \leq 1}\left|q'(y)\right|=2$ and $q''(y)=-2$ for $0\leq y \leq 1$. Hence, from \eqref{sharp} it follows that
  \begin{align}
    &\left|\overline{ v}_{1yy} - \overline{ v}_{2y y}\right|_{L^2(0,1)}\notag\\
    &= \left( \int^1_0 \left| \frac{e_1^2}{k_l}c_l\overline{v}_{1t} + 2\overline{u}_{1y}q' + \overline{u}_{1}q'' - \frac{c_l y e'_1 e_1}{k_l} \overline{v}_{1y} +  \frac{c_l y e'_1 e_1}{k_l} \overline{u}_1 q' \right. \right.\notag\\
    &\quad \, \left.\left. - \frac{e_2^2}{k_l}c_l\overline{v}_{2t} - 2\overline{u}_{2y}q' - \overline{u}_{2}q'' + \frac{c_l y e'_2 e_2}{k_l} \overline{v}_{2y} -  \frac{c_l y e'_2 e_2}{k_l} \overline{u}_2 q' \right|^2 dy \right)^\frac{1}{2}\notag\\
    &\leq R_4 \left( (e_1^2 - e_2^2)^2 \int^1_0 \left|\overline{u}_{1t}\right|^2 dy +  \int^1_0\left|\overline{v}_{1t} - \overline{v}_{2t}\right|^2dy+  \int^1_0 \left|\overline{u}_{1y} - \overline{u}_{2y}\right|^2dy \right.\notag\\
    & \quad \,  +  \int^1_0 \left|\overline{u}_{1} - \overline{u}_{2}\right|^2dy+  \left|e'_2\right|^2  \int^1_0 \left|\overline{v}_{1y} - \overline{v}_{2y}\right|^2dy\notag\\
   & \quad \,   +   \left|e'_2\right|^2 \left|e_1 - e_2\right|^2 \int^1_0 \left|q'\overline{u}_1 + q \overline{u}_{1y}\right|^2dy  +   \left|e'_1 - e'_2\right|^2 \int^1_0 \left|\overline{v}_{1y}\right|^2dy\notag \\
    &\left. \quad \,+  \left|e'_2\right|^2 \int^1_0 \left|\overline{u}_{1} - \overline{u}_{2}\right|^2dy+    \left|e'_2\right|^2 \left|e_1 - e_2\right|^2 \int^1_0 \left|\overline{u}_{1}\right|^2dy +   \left|e'_1 - e'_2\right|^2 \int^1_0 \left|\overline{u}_{1}\right|^2dy \right)^\frac{1}{2}\notag\\
&\hspace{80mm} \text{a.e. on } [0,T],\label{bay}
 \end{align}
 where $R_4$ is a  positive constant. Integrating \eqref{bay} with respect to the time variable, by Lemma \ref{lemma100}, \eqref{vr2} and \eqref{vr3} we see that
  \begin{align*}
  \int^t_0 \int^1_0 \left|\overline{ v}_{1yy} - \overline{ v}_{2y y}\right|^2 dy d\tau \leq R_5 \int^{t}_0 (\left|e_{1 }- e_{2 }\right|^2 + \left|e_{1 }' - e_{2 }'\right|^2 ) d\tau \text{ for } 0 \leq t \leq T,
  \end{align*} 
   where $R_5$ is a  suitable positive constant.
Since \[ \displaystyle u_i(t,x) = \overline{u}_i\left(t, \frac{x}{e_i(t)}\right)  \text{for } (t,x) \in Q_l(T,e_i)\] and $i = 1,2$, we have
 \begin{align}
&\left| u_{1x}(t, e_1(t)-) - u_{2x}(t, e_2(t)-)\right|= \left| \frac{1}{e_1(t)} \overline{ u}_{1y}(t, 1-) - \frac{1}{e_2(t)}  \overline{ u}_{2 y}(t, 1-)\right|\notag\\
&\leq \left| \left( \frac{1}{e_1(t)} - \frac{1}{e_2(t)} \right) \overline{ u}_{1y}(t, 1-)  \right| +  \left| \frac{1}{e_2(t)} \left(\overline{ u}_{1y}(t, 1-) - \overline{ u}_{2 y}(t, 1-)\right) \right|\notag\\
&=: I_1(t) + I_2(t) \text{ for a.e. } t \in [0,T]. \label{ham}
 \end{align}
 By \eqref{maru}, we see that
  \begin{align}
  I_1(t) &\leq \frac{e_1(t)}{\delta^2}\left|e_1(t) - e_2(t)\right|| u_{1x}(t, e_1(t)-)| \leq \frac{C_2}{\delta^2}\left|e_1(t) - e_2(t)\right| \text{ for a.e. } t \in [0,T]. \label{seibu}
   \end{align}
   Also, thanks to the definition of $ \overline{v}_{i }$ and Lemma \ref{lemgag}, we infer that $\overline{u}_{iy}(t,1-) = \overline{v}_{iy}(t,1-)$ for a.e. $t \in [0,T]$ and $i = 1, 2$, and
    \begin{align}
  I_2(t) &\leq  \frac{1}{\delta}  \left|\overline{ v}_{1y}(t, 1-) - \overline{ v}_{2 y}(t, 1-)\right|\notag\\
& \leq \frac{C_0}{\delta}\left(  \left|\overline{ v}_{1yy}(t) - \overline{ v}_{2y y}(t)\right|_{L^2(0,1)}^{1/2} \left|\overline{ v}_{1y}(t) - \overline{ v}_{2y}(t)\right|_{L^2(0,1)}^{1/2} + \left|\overline{ v}_{1y}(t) - \overline{ v}_{2y}(t)\right|_{L^2(0,1)} \right)\notag\\
&\hspace{80mm}  \text{ for a.e. } t \in [0,T]. \label{lions}
   \end{align}
On account of  \eqref{vr2}-\eqref{lions} and Lemma \ref{lemma100}, we can choose  positive constants $R_6$ and $R_7$ depending only on $\delta, M, T$ and $\left|u_0\right|_{W^{1,\infty}(0,1)}$ such that
\begin{multline*}
\int^t_0 \left|u_{1x}(\tau,e_1(\tau)-) - u_{2x}(\tau,e_2(\tau)-)\right|^3 d\tau \\
\shoveleft{\hspace{1.5cm}\leq  \left(\frac{C_2}{\delta^2}\right)^3 \int^t_0 \left|e_1(\tau) - e_2(\tau)\right|^3 d\tau} \\
\shoveleft{\hspace{3cm}+  \left( \frac{2 C_0}{\delta} \right)^3 \int^t_0 \Big(  \left|\overline{ v}_{1yy}(\tau) - \overline{ v}_{2y y}(\tau)\right|_{L^2(0,1)}^{3/2}
\left|\overline{ v}_{1y}(\tau) - \overline{ v}_{2y}(\tau)\right|_{L^2(0,1)}^{3/2}} \\
\shoveleft{\hspace{4.5cm}+ \left|\overline{ v}_{1y}(\tau) - \overline{ v}_{2y}(\tau)\right|_{L^2(0,1)}^3  \Big) d\tau} \\
\shoveleft{\hspace{1.5cm}\leq  \left(\frac{C_2}{\delta^2}\right)^3 \int^t_0 \left|e_1(\tau) - e_2(\tau)\right|^3 d\tau} \\
\shoveleft{\hspace{3cm}+  \left( \frac{2 C_0}{\delta} \right)^3  \Big(  \left|\overline{ v}_{1y}- \overline{ v}_{2 y}\right|_{L^{\infty}(0,t;L^2(0,1))}^{3/2}
\int^t_0 \left|\overline{ v}_{1yy}(\tau) - \overline{ v}_{2yy}(\tau)\right|_{L^2(0,1)}^{3/2} d\tau} \\
\shoveleft{\hspace{4.5cm}+ \left|\overline{ v}_{1y} - \overline{ v}_{2y}\right|_{L^{\infty}(0,t;L^2(0,1))}^3 t  \Big)} \\
\shoveleft{\hspace{1.5cm}\leq  \left(\frac{C_2}{\delta^2}\right)^3 t^3 \int^t_0 \left|e'_1 - e'_2\right|^3 d\tau} \\
\shoveleft{\hspace{3cm}+  \left( \frac{2 C_0}{\delta} \right)^3  \Big(  \left|\overline{ v}_{1y}- \overline{ v}_{2 y}\right|_{L^{\infty}(0,t;L^2(0,1))}^{3/2}
\left|\overline{ v}_{1yy}- \overline{ v}_{2 yy}\right|_{L^{2}(0,t;L^2(0,1))}^{3/2} t^{1/4}}\\
\shoveleft{\hspace{4.5cm}+ \left|\overline{ v}_{1y} - \overline{ v}_{2y}\right|_{L^{\infty}(0,t;L^2(0,1))}^3 t  \Big)} \\
\leq  R_6 t^3 \int^t_0 \left|e'_1 - e'_2\right|^3 d\tau +  R_6\left(  \int^{t}_0 (\left|e_{1 }- e_{2 }\right|^2 + \left|e_{1 }' - e_{2 }'\right|^2 ) d\tau \right)^{3/2}t^{1/4} \text{ for } 0 \leq t \leq T.
\end{multline*}\\
Hence, it holds that
\begin{align*}
\left|u_{1x}(\cdot,e_1-) - u_{2x}(\cdot,e_2-)\right|_{L^3(0,t)}^3\leq  R_7 t^{1/4}  \int^{t}_0 (\left|e_{1 }- e_{2 }\right|^3 + \left|e_{1 }' - e_{2 }'\right|^3 ) d\tau \text{ for } 0 \leq t \leq T.
\end{align*}
Similarly to above, we get
\begin{align*}
& \int^t_0 \left|u_{1x}(\tau,e_1(t)+) - u_{2x}(\tau ,e_2(\tau)+)\right|^3 d\tau\\
& \leq R_{8}  t^{1/4} \int^{t}_0 (\left|e_{1} - e_{2}\right|^3 + \left|e'_{1} - e'_{2}\right|^3) d\tau \text{ for } 0 \leq t \leq T, 
\end{align*}
where $R_8$ is a positive constant. Thus, Lemma \ref{lemma101} has been proved.
\end{proof}
\section{Proof of Proposition \ref{prop1}}
\label{sec3}
To show  Proposition \ref{prop1}, we introduce the function $\zeta \in  C^{\infty}([0,1])$ which satisfies 
 $0 < \zeta \leq 1  \text{ on } (1-\delta,1]$,
   $\zeta = 0  \text{ on } [0,1-\delta]$,
$\zeta(1)=1$, and $\zeta'(1-\delta) = \zeta'(1)=0$.
Let $e \in K(\delta, e_{0}, M,T)$ and $u$ be a solution of AP1($e,u_0$) on $[0,T]$ for $T > 0$, and put $\tilde{u} = \zeta u$. Then, the following equation, the boundary condition and the initial condition hold:
\begin{align}
c_a \tilde{u}_t &= k_a (\tilde{u}_{xx} - \zeta_{xx}u - 2 \zeta_x u_x) \text{ in } (0,T) \times (1-\delta, 1),\label{naga}\\
- k_a \frac{\partial \tilde{u}}{\partial x}(t,1) &= h( \tilde{u}(t,1) + \theta - u_b(t)) + \sigma( (\tilde{u}(t,1) + \theta)^4 - u_b(t)^4)\notag\\
&=g(t, \tilde{u}(t,1)) \text{ for } 0 < t < T,\notag\\
\tilde{u}(t,1-\delta) &= 0 \text{ for } 0 < t < T,\notag\\
\tilde{u}(0,x) &= \tilde{u}_0(x) :=(\zeta u_0)(x) \text{ for } 1-\delta \leq x \leq 1,\notag
\end{align}
where $g$ is the function defined in Section \ref{sec2}. The next lemma guarantees the existence of smooth approximations to $u_0$ and $u$, which will be used in the proof of Theorem \ref{th1}.
\begin{lemma}\label{lem1}

Suppose that a pair ($u_0, e_0$) satisfies \Aref{5} and let $0 < \delta < 1$, $\delta \leq e_0 \leq 1 - \delta$, $M > 0$ and $T > 0$. If $u$ is a solution of AP1($e, u_0$) for $e \in K(\delta, e_{0}, M,T)$, then there exist approximate sequences $\{u_{0n} \}$ and $\{u_{n}\}$ to $u_0$ and $u$, respectively, such that
\begin{align*}
&\{u_{0 n} \} \subset C^{\infty}([1-\delta,1]), u_{0 n} \to u_0 \text{ in } W^{2,2}(1-\delta,1) \text{ as } n \to \infty,\\
&\hspace{1.5cm}-k_a u_{0 n x}(1) =  h( u_{0 n}(1) + \theta - u_b(0)) + \sigma( (u_{0 n}(1) + \theta)^4 - u_b(0)^4), \text{ and}\\
&\{u_{n}\} \subset W^{1,2}(0,T;W^{2,2}(1-\delta,1)) \text{ for any } n \in \mathbb{N},\\
&u_{n} \to u \text{ in } L^{2}(0,T;W^{2,2}(1-\delta,1))  \text{ as } n \to \infty,\\
& u_{n t} \to u_t \text{ in } L^{2}(0,T;L^{2}(1-\delta,1))  \text{ as } n \to \infty,\\
&u_{n}(t) = u_{0 n} \text{ on } [1-\delta,1]  \text{ for } t \leq 0.
\end{align*}
\end{lemma}
\begin{proof}
First, we put $u(t) = u_0$ for $t < 0$ and $u_{\nu}(t) = \frac{1}{\nu}\int^{t}_{t - \nu}u d\tau$ in $L^2(1-\delta, 1)$  for $t \leq T$ and  $\nu > 0$. Easily, we have $u_{\nu}(t) = u_0$ for $t \leq 0$, 
\[u_{\nu x} = \frac{1}{\nu}\int^{t}_{t - \nu}u_x d\tau \text{ and  }u_{\nu xx} = \frac{1}{\nu}\int^{t}_{t - \nu}u_{xx} d\tau \text{ in }L^2(1-\delta, 1) 
\text{ on } [0,T],\] since $u \in L^2(0,T; W^{2,2}(1-\delta, 1))$. For $t, t' \in [0,T]$ it is easy to see that
\begin{align*}
\left|u_{\nu}(t)- u_{\nu}(t')\right|_{L^2(1-\delta, 1)} 
\leq \frac{1}{\nu}\int^{0}_{- \nu}  \left|  u(t + \tau) -  u(t' + \tau)  \right|_{L^2(1-\delta, 1)} d\tau \to 0 \text{ as } t' \to t,
\end{align*}
and it implies $u_{\nu} \in C([0,T]; L^2(1-\delta,1))$. \\
Similarly, 
we can show that $u_{\nu} \in C([0,T]; W^{2,2}(1  -\delta,1))$. Also, we obtain
\begin{align*}
&\int^{T}_{0}  \left|  u_{\nu}(t ) -  u(t)  \right|_{L^2(1-\delta, 1)}^2 dt \leq \frac{1}{\nu} \int^{T}_{0}  \left( \int^{0}_{ - \nu} \left| u(t + \tau) - u(t) \right|_{L^2(1-\delta, 1)} d\tau   \right)^2 dt\\
&\leq  \int^{0}_{ - \nu} \int^{T}_{0}   \left| u(t + \tau) - u(t) \right|_{L^2(1-\delta, 1)}^2 dt d\tau \text{ for } \nu > 0.
\end{align*}
Consequently, we see that $u_{\nu} \to u$ in $L^2(0,T; L^2(1-\delta, 1))$ as $\nu \to 0$.
Similarly, we have $u_{\nu} \to u$  in $L^2(0,T; W^{2,2}(1-\delta, 1))$ as $\nu \to 0$. Also, it holds that $u_{\nu t} \to u_t$ in $L^2(0,T;L^2(1-\delta, 1))$ as $\nu \to 0$. Moreover, it is clear that $u_{\nu}(0) = u_0 \geq 0$ on $[1-\delta, 1]$ and by \Aref{5} $-k_a u_{\nu x}(0,1) = g(0, u_{\nu}(0,1))$. Therefore, \[ -k_a u_{\nu x}(t,1) = g(t, u_{\nu}(t,1)) \text{ for } t \in \left(-\infty, T\right]. \] 

Here, we note that $u_{\nu t} \in L^2(0,T;W^{2,2}(1-\delta, 1))$ for $\nu > 0$.
Next, let $J_{\varepsilon}$ be a mollifier in $\mathbb{R}$ for $\varepsilon > 0$. For $\varepsilon > 0$ and $\nu > 0$ we put
\begin{align*}
&U^0_{\nu \varepsilon}(t, x) = \int_{\mathbb{R}} u_{\nu xx}(t,y) J_{\varepsilon}(x-y)dy,\\
&U^1_{\nu \varepsilon}(t, x) = - \int^1_{x} U^0_{\nu \varepsilon}(t, \xi) d\xi + u_{\nu x}(t,1),\\
&u_{\nu \varepsilon}(t, x) = - \int^1_{x} U^1_{\nu \varepsilon}(t, \xi) d\xi + u_{\nu}(t,1) \text{ for } (t,x) \in (- \infty,T) \times (1-\delta, 1).
\end{align*}
By $u_{\nu xxt} \in L^2(0,T;L^2(1-\delta, 1))$, we see that
\begin{align*}
U^0_{\nu \varepsilon t} &= J_{\varepsilon} \ast u_{\nu xxt} \text{ on } (0,T) \times (1-\delta, 1),\\
U^1_{\nu \varepsilon t}(t, x) &= - \int^1_{x} U^0_{\nu \varepsilon t}(t, \xi) d\xi + u_{\nu xt}(t,1) \text{ and }
u_{\nu \varepsilon t}(t, x) = - \int^1_{x} U^1_{\nu \varepsilon t}(t, \xi) d\xi + u_{\nu t}(t,1)
\end{align*}
for $(t, x) \in (0,T) \times (1-\delta,1)$. Also, it holds that $u_{\nu \varepsilon}(t) = u_{\nu \varepsilon}(0)$ for $t \leq 0$. Indeed, since $U^0_{\nu \varepsilon }(t) = J_{\varepsilon} \ast u_{\nu xx}(t) = J_{\varepsilon} \ast u_{0xx}$ for $t \leq 0$, we have
\begin{align*}
U^1_{\nu \varepsilon}(t, x) &= - \int^1_{x}  (J_{\varepsilon} \ast u_{0xx}) d\xi + u_{0 x}(1),\text{ and }
u_{\nu \varepsilon}(t, x) = u_{\nu \varepsilon}(0, x) 
\end{align*}
for $(t,x) \in (- \infty,0] \times [0, 1]$.
\par Moreover, it is clear that we have  $U^0_{\nu \varepsilon}(t) \in C^{\infty}([1-\delta,1])$ for $t \in [0,T]$ and 
\begin{align*}
\left|U^0_{\nu \varepsilon}(t) - U^0_{\nu \varepsilon}(t')\right|_{L^2(1-\delta, 1)} \leq \left|u_{\nu xx}(t) - u_{\nu xx}(t')\right|_{L^2(1-\delta, 1)} \text{ for } t, t' \in [0,T].
\end{align*}
Since $u_{\nu} \in C([0,T]; W^{2,2}(1-\delta, 1))$, we have $U^0_{\nu \varepsilon} \in C([0,T]; L^2(1-\delta, 1))$. Easily, we get 
\[ \left|U^0_{\nu \varepsilon}(t)\right|_{L^2(1-\delta, 1)} \leq \left|u_{\nu xx}(t)\right|_{L^2(1-\delta, 1)} \text{ and } U^0_{\nu \varepsilon}(t) \to u_{\nu xx}(t)\] in $L^2(1-\delta, 1)$ as $\varepsilon \to 0$ for $0 \leq t \leq T$. By applying the dominated convergence theorem,  we get $U^0_{\nu \varepsilon} \to u_{\nu xx}$ in $L^2(0,T; L^2(1-\delta, 1))$ as $\varepsilon \to 0$. Clearly $u_{\nu \varepsilon x} = U^1_{\nu \varepsilon}$ and $u_{\nu \varepsilon xx} = U^0_{\nu \varepsilon}$ on $(0,T) \times (1-\delta, 1)$, that is, $u_{\nu \varepsilon} \in C([0,T]; W^{2,2}(1-\delta, 1))$ for any $\nu$, $\varepsilon > 0$. Also, we have
\begin{align*}
|u_{\nu \varepsilon x}(t, x) - u_{\nu x}(t, x)| &= \left| \int^x_1 \frac{\partial}{\partial \xi}(u_{\nu \varepsilon x}(t, \xi) - u_{\nu x}(t, \xi))d\xi \right|\\
& \leq \int^1_{1-\delta} |U^0_{\nu \varepsilon}(t) - u_{\nu xx}(t)| d\xi \text{ for a.e. } t \in [0,T].
\end{align*}
This shows
\begin{align*}
u_{\nu \varepsilon x} \to  u_{\nu x} \text{ in } L^2(0,T;L^2(1-\delta, 1)) \text{ as } \varepsilon \to 0.
\end{align*}
Similarly, we have
\begin{align*}
u_{\nu \varepsilon} \to  u_{\nu} \text{ in } L^2(0,T;L^2(1-\delta, 1)) \text{ as } \varepsilon \to 0.
\end{align*}
Easily, we see that
\[ -k_a u_{\nu \varepsilon x}(0,1)= -k_a u_{\nu  x}(0,1) = g(0, u_{\nu}(0,1)) = g(0, u_{\nu \varepsilon}(0,1)).\]
 Also, we obtain
\begin{align*}
|u_{\nu \varepsilon t}(t, x) - u_{\nu t}(t, x)| &= \left| \int^1_x (u_{\nu \varepsilon tx}(t, \xi) - u_{\nu tx}(t, \xi)) d\xi \right|\\
& \leq  \int^1_0  \left| u_{\nu \varepsilon tx}(t, x) - u_{\nu tx}(t, x)\right| dx\\
& \leq  \int^1_0 \int^1_x  \left| u_{\nu \varepsilon txx}(t, \xi) - u_{\nu txx}(t, \xi)\right| d\xi dx\\
&\leq  \left| u_{\nu \varepsilon txx}(t) - u_{\nu txx}(t)\right|_{L^2(1-\delta, 1)}\\
&= \left|  (J_{\varepsilon} \ast u_{\nu txx})(t) - u_{\nu txx}(t)\right|_{L^2(1-\delta, 1)}\\
&\hspace{40mm} \text{for }  (t,x) \in (0,T) \times (1-\delta, 1).
\end{align*}
Hence, thanks to the basic property of the mollifier we have
\begin{align*}
u_{\nu \varepsilon t} \to  u_{\nu t} \text{ in } L^2(0,T;L^2(1-\delta, 1)) \text{ as } \varepsilon \to 0 \text{ for } \nu > 0.
\end{align*}
Accordingly, for any $n \in \mathbb{N}$, there exist $\nu_n > 0$ and $\varepsilon_n > 0$ such that
\[\left|u_{\nu_n} - u\right|_{ L^2(0,T;W^{2,2}(1-\delta, 1))} \leq \frac{1}{2n}, \left|u_{\nu_n t} - u_t\right|_{ L^2(0,T;L^2(1-\delta, 1))}  \leq \frac{1}{2n},\]
\[\left|u_{\nu_n} - u_{\nu_n \varepsilon_n}\right|_{ L^2(0,T;W^{2,2}(1-\delta, 1))} \leq \frac{1}{2n} \text{ and } \left|u_{\nu_n  \varepsilon_n t} - u_{\nu_n t}\right|_{ L^2(0,T;L^2(1-\delta, 1))}  \leq \frac{1}{2n}.\]
By putting $u_{n} = u_{\nu_n \varepsilon_n}$ for each $n$,  we have
\begin{align*}
&u_n \to u \text{ in } L^2(0,T;W^{2,2}(1-\delta, 1)),\\
&u_{nt} \to u_t \text{ in } L^2(0,T;L^2(1-\delta, 1)) \text{ as } n \to \infty.
\end{align*}
Easily, we get $u_n(0) = u_{\nu_n \varepsilon_n}(0) \in C^{\infty}([1-\delta, 1])$ and $-k_a u_{n x}(0,1) =g(0, u_{n}(0,1))$ for each $n$. By putting $u_{0n} = u_n(0)$ for $n$, we can prove Lemma \ref{lem1}.
\end{proof}
For any $n$ let \[\tilde{u}_n \in W^{1,2}(0, T; L^2(1-\delta, 1)) \cap L^{\infty}(0, T; W^{1,2}(1-\delta,1))\] be a solution of the following problem:
\begin{align}
c_a \tilde{u}_{n t} &= k_a (\tilde{u}_{n xx} - \zeta_{xx}u_{n} - 2 \zeta_x u_{ n x}) \text{ in } (0,T) \times (1-\delta, 1),\label{first}\\
- k_a \frac{\partial \tilde{u}_{n}}{\partial x}(t,1) &= h( [\tilde{u}_{n}(t,1) + \theta_c]^+ - u_b(t)) + \sigma( ([\tilde{u}_{n}(t,1) + \theta_c]^+)^4 - u_b(t)^4)\notag\\
&=: g_+(t, \tilde{u}_n(t,1)) \text{ for } 0 < t < T,\label{second}\\
\tilde{u}_{n}(t,1-\delta) &= 0 \text{ for } 0 < t < T,\label{third}\\
\tilde{u}_{n}(0,x) &= \tilde{u}_{0 n}(x)=(\zeta u_{0 n })(x) \text{ for } 1- \delta \leq x \leq 1.\label{forth}
\end{align}
Since  the equation \eqref{first} is linear and the boundary condition \eqref{second} is monotone, for given $u_n$ we can prove existence and uniqueness of the strong solution to the problem \eqref{first}-\eqref{forth} in the similar way to AP1$(e, u_0)$.
\begin{lemma}\label{lem3}
 Under \Aref{4} and \Aref{5}, let $\tilde{u}_n$ be the strong solution of \eqref{first}-\eqref{forth}. Then, there exists a constant $C > 0$ independent of $n$ such that $\left|\tilde{u}_n(t,1)\right| \leq C$ for $0 \leq t \leq T$.
\end{lemma}
\begin{proof}
First, by multiplying $\tilde{u}_{n}$ on both sides of  \eqref{first} and integrating it with $x$ on $[1-\delta, 1]$, we have
\begin{align*}
\int^1_{1-\delta} c_a \tilde{u}_{nt}\tilde{u}_n dx&=k_a \int^1_{1-\delta}  \tilde{u}_{nxx} \tilde{u}_ndx - k_a \int^1_{1-\delta} \zeta_{xx}  u_n \tilde{u}_n dx - 2 k_a \int^1_{1-\delta} \zeta_{x}  u_{nx} \tilde{u}_ndx\\
&=: I_1 + I_2 + I_3 \text{ a.e. on } [0,T].
\end{align*}
By elementary calculations, we obtain
\[\int^1_{1-\delta} c_a \tilde{u}_{nt}\tilde{u}_n dx = \frac{d}{dt}\frac{c_a}{2} \int^1_{1-\delta} \tilde{u}_n^2 dx,
\]
and
\[I_1 = -g_+(\cdot, \tilde{u}_n(\cdot,1))\tilde{u}_n(\cdot,1) - k_a \int^1_{1-\delta} \left|\tilde{u}_{nx}\right|^2 dx \text{ a.e. on } [0,T].
\]
Thanks to Young's inequality, we see that
\begin{align*}
I_2 \leq \frac{c_a}{4} \int^1_{1-\delta} \tilde{u}_n^2 dx + \frac{k_a^2 }{c_a}\left|\zeta_{xx}\right|^2_{L^{\infty}(0,1)} \int^1_{1-\delta} u_n^2 dx,
\end{align*}
and
\begin{align*}
I_3 \leq \frac{c_a}{4} \int^1_{1-\delta} \tilde{u}_n^2 dx + \frac{4 k_a^2}{c_a}  \left|\zeta_{x}\right|^2_{L^{\infty}(0,1)} \int^1_{1-\delta} \left|u_{nx}\right|^2 dx \text{ a.e. on } [0,T].
\end{align*}
Hence, we have
\begin{align*}
& \frac{d}{dt}\frac{c_a}{2} \int^1_{1-\delta} \tilde{u}_n^2 dx + k_a \int^1_{1-\delta} \left|\tilde{u}_{nx}\right|^2 dx\\
& \leq  -g_+(\cdot, \tilde{u}_n(\cdot,1))\tilde{u}_n(\cdot,1) + \frac{c_a}{2} \int^1_{1-\delta} \tilde{u}_n^2 dx\\
&\quad \, + \frac{k_a^2}{c_a} \left|\zeta_{xx}\right|^2_{L^{\infty}(0,1)} \int^1_{1-\delta} u_n^2 dx + \frac{4 k_a^2}{c_a} \left|\zeta_{x}\right|^2_{L^{\infty}(0,1)} \int^1_{1-\delta} \left|u_{nx}\right|^2 dx \text{ a.e. on } [0,T].
\end{align*}
Here, we estimate the first term of the above inequality by the monotonicity and Young's inequality. Namely, we have
\begin{align*}
& -g_+(\cdot, \tilde{u}_n(\cdot,1))\tilde{u}_n(\cdot,1)\\
&= -h [\tilde{u}_{n}(\cdot,1) + \theta_c]^+(\tilde{u}_n(\cdot,1) + \theta_c - \theta_c) +h u_b\tilde{u}_n(\cdot,1)\\
&\quad \,  - \sigma(([\tilde{u}_{n}(\cdot,1) + \theta_c]^+)^4 - \theta_c^4 )(\tilde{u}_n(\cdot,1) + \theta_c-\theta_c)-\sigma \theta_c^4 \tilde{u}_n(\cdot,1) + \sigma u_b^4\tilde{u}_n(\cdot,1)\\
&\leq h [\tilde{u}_{n}(\cdot,1) + \theta_c]^+\theta_c +h u_b\tilde{u}_n(\cdot,1)  -\sigma \theta_c^4 \tilde{u}_n(\cdot,1) + \sigma u_b^4\tilde{u}_n(\cdot,1)\\
&\leq C_1|\tilde{u}_{n}(\cdot,1)| + h \theta_c^2 \text{ a.e. on } [0,T],
\end{align*}
where $\displaystyle C_1 = h(\theta_c + \max_{0 \leq t \leq T}|u_b|) + \sigma(\theta_c^4 + \max_{0 \leq t \leq T}u_b^4)$. Moreover, Lemma \ref{lemgag} guarantees existence of a positive constant $C_0$ depending only on $\delta$ such that 
\begin{align*}
& -g_+(\cdot, \tilde{u}_n(\cdot,1))\tilde{u}_n(\cdot,1)\\
&\leq C_0 C_1  \left|\tilde{u}_{nx}\right|^{1/2}_{L^{2}(1-\delta,1)} \left|\tilde{u}_{n}\right|^{1/2}_{L^{2}(1-\delta,1)} + h \theta_c^2\\
&\leq \frac{k_a}{2}  \left|\tilde{u}_{nx}\right|^{2}_{L^{2}(1-\delta,1)} + C_2 ( \left|\tilde{u}_{n}\right|^{2}_{L^{2}(1-\delta,1)} + 1) \text{ a.e. on } [0,T],
\end{align*}
where $C_2$ is a positive constant independent of $n$. Thus, it follows that
\begin{align*}
& \frac{d}{dt}\frac{c_a}{2} \int^1_{1-\delta} \tilde{u}_n^2 dx + \frac{k_a}{2} \int^1_{1-\delta} \left|\tilde{u}_{nx}\right|^2 dx\\
& \leq  C_3\left( \int^1_{1-\delta} \tilde{u}_n^2 dx + \int^1_{1-\delta} u_n^2 dx + \int^1_{1-\delta} \left|u_{nx}\right|^2 dx \right) \text{ a.e. on } [0,T],
\end{align*}
where $C_3$ is a positive constant independent of $n$.
By applying Gronwall's inequality to above, there exists a positive constant $C_4$ such that
\begin{align*}
\frac{c_a}{2} \int^1_{1-\delta} \tilde{u}_n(t,x)^2 dx + \frac{k_a}{2}\int^t_0 \int^1_{1-\delta} \left|\tilde{u}_{nx}(\tau,x)\right|^2 dxd\tau \leq C_4 \text{ for } 0 \leq t \leq T.
\end{align*}
\par Next, by multiplying $\tilde{u}_{nt}$ on both sides of  \eqref{first} and integrating it with $x$ on $[1-\delta, 1]$, we have
\begin{align*}
\int^1_{1-\delta} c_a \tilde{u}_{nt}^2 dx&=k_a \int^1_{1-\delta}  \tilde{u}_{nxx} \tilde{u}_{nt} dx - k_a \int^1_{1-\delta} \zeta_{xx}  u_n \tilde{u}_{nt} dx\\
&\quad \,  - 2 k_a \int^1_{1-\delta} \zeta_{x}  u_{nx} \tilde{u}_{nt} dx \text{ a.e. on } [0,T].
\end{align*}
By integration by parts for the first term of the right hand side and applying Young's inequality, we see that
\begin{align*}
&c_a \int^1_{1-\delta}  \left|\tilde{u}_{nt}(t)\right|^2 dx + \frac{d}{dt}\left( \frac{k_a}{2}\int^1_{1-\delta}  \left|\tilde{u}_{nx}(t)\right|^2  dx + \hat{g}_+(t, \tilde{u}_n(t,1)) \right) \\
&= \frac{\partial \hat{g}_+}{\partial t}(t, \tilde{u}_n(t,1)) - k_a \int^1_{1-\delta} \zeta_{xx}  u_n(t) \tilde{u}_{nt}(t) dx - 2 k_a \int^1_{1-\delta} \zeta_{x}  u_{nx}(t) \tilde{u}_{nt}(t)dx\\
&\leq \frac{\partial \hat{g}_+}{\partial t}(t, \tilde{u}_n(t,1)) + \frac{c_a}{2}\int^1_{1-\delta}  \left|\tilde{u}_{nt}(t)\right|^2 dx + \frac{k_a^2}{c_a}\left|\zeta_{xx}\right|_{L^{\infty}(1-\delta, 1)}^2 \int^1_{1-\delta}u_n(t)^2 dx\\
&\quad \,  + \frac{4 k_a^2}{c_a}\left|\zeta_{x}\right|_{L^{\infty}(1-\delta, 1)}^2 \int^1_{1-\delta}\left|u_{nx}(t)\right|^2 dx \text{ for a.e. } t \in [0,T],
\end{align*}
where $\hat{g}_+(t, r) = \int^r_0 g_+(t,\xi)d\xi$ for $r \in \mathbb{R}$. Here, we note that the differentiability of the function \[t \to \frac{k_a}{2}\int^1_{1-\delta}  \left|\tilde{u}_{nx}(t)\right|^2  dx + \hat{g}_+(t, \tilde{u}_n(t,1))\] can be proved in the similar way to that of Lemma 4.4 in \cite{AK}. 
By Lemma \ref{lem1} we can take a positive constant $C_5$ independent of $n$ such that 
\begin{align*}
&\frac{c_a}{2} \int^1_{1-\delta}  \left|\tilde{u}_{nt}(t)\right|^2 dx + \frac{d}{dt} \left( \frac{k_a}{2}\int^1_{1-\delta}  \left|\tilde{u}_{nx}(t)\right|^2 dx  + \hat{g}_+(t, \tilde{u}_n(t,1)) \right)\\
&\leq  \frac{\partial \hat{g}_+}{\partial t}(t, \tilde{u}_n(t,1)) + C_5\\
&= -hu'_b(t)\tilde{u}_n(t,1) - 4\sigma u_b^3(t) u'_b(t) \tilde{u}_n(t,1) + C_5\\
&\leq C_6|u'_b(t)|\left( \int^1_{1-\delta}  \tilde{u}_{n}(t)^2 dx + \int^1_{1-\delta}  \left|\tilde{u}_{nx}(t)\right|^2 dx  \right) + C_5 \text{ for a.e. } t \in  [0,T],
\end{align*}
where $C_6$ is a positive constant. Since $\hat{g}_+(t, r) \geq 0$ for $t \in [0,T]$ and $r \geq 0$, from Gronwall's inequality, it follows that
\begin{align*}
&\frac{c_a}{2}\int^t_0 \int^1_{1-\delta}  \left|\tilde{u}_{nt}(\tau,x)\right|^2 dxd\tau + \frac{k_a}{2}\int^1_{1-\delta}  \left|\tilde{u}_{nx}(t,x)\right|^2 dx  + \hat{g}_+(t, \tilde{u}_n(t,1)) \\
&\leq C_7 \text{ for } 0 \leq t \leq T,
\end{align*}
where $C_7$ is a positive constant.
Since $\tilde{u}_n \in L^{\infty}(0, T; W^{1,2}(1-\delta,1))$ and $\tilde{u}_n(t, 1-\delta) = 0$, on account of Lemma 4.4 we can prove Lemma \ref{lem3}, completely.
\end{proof}

\begin{lemma}\label{lem2}
 Under \Aref{4} and \Aref{5}, let $\tilde{u}_n$ be the strong solution of \eqref{first}-\eqref{forth} for $n$ and $\tilde{u} = \zeta u$ on $(0,T) \times (0,1)$. Then, $\tilde{u}_n(\cdot,1) \to \tilde{u}(\cdot,1) \text{ in } L^2(0,T) \text{ as } n \to \infty$.
\end{lemma}
\begin{proof}
By multiplying $\tilde{u}_n - \tilde{u}$ on both sides of difference of \eqref{first} and \eqref{naga} for $n$, we have
\begin{align*}
&\int^1_{1-\delta} c_a (\tilde{u}_{nt} - \tilde{u}_t)(\tilde{u}_n - \tilde{u})dx\\
&=k_a \int^1_{1-\delta}  (\tilde{u}_{nxx} - \tilde{u}_{xx})(\tilde{u}_n - \tilde{u})dx  - k_a \int^1_{1-\delta} \zeta_{xx}  (u_n - u)(\tilde{u}_n - \tilde{u})dx\\
&\quad \, - 2 k_a \int^1_{1-\delta} \zeta_{x}  (u_{nx} - u_x)(\tilde{u}_n - \tilde{u})dx =: I_1 + I_2 + I_3 \text{ a.e. on } [0,T].
\end{align*}
On the left hand side, we see that
\begin{align*}
\int^1_{1-\delta} c_a (\tilde{u}_{nt} - \tilde{u}_t)(\tilde{u}_n - \tilde{u})dx = \frac{d}{dt} \frac{c_a}{2} \int^1_{1-\delta} \left|\tilde{u}_n - \tilde{u}\right|^2 dx \text{ a.e. on } [0,T].
\end{align*}
Easily, we observe that
\begin{align*}
I_1 &= k_a (\tilde{u}_{nx}(\cdot, 1) - \tilde{u}_{x}(\cdot, 1))(\tilde{u}_n(\cdot, 1) - \tilde{u}(\cdot, 1))\\
&\quad \, - k_a \int^1_{1-\delta}  \left|\tilde{u}_{nx} - \tilde{u}_{x}\right|^2 dx \text{ a.e. on } [0,T].
\end{align*}
Thanks to Young's inequality, we have
\begin{align*}
I_2 \leq \frac{c_a}{4} \int^1_{1-\delta}  \left|\tilde{u}_{n} - \tilde{u}\right|^2 dx + C_1 \int^1_{1-\delta} \left|u_n - u\right|^2 dx \text{ a.e. on } [0,T],
\end{align*}
and
\begin{align*}
I_3 \leq \frac{c_a}{4} \int^1_{1-\delta}  \left|\tilde{u}_{n} - \tilde{u}\right|^2 dx + C_2 \int^1_{1-\delta} \left|u_{nx} - u_x\right|^2 dx \text{ a.e. on } [0,T],
\end{align*}
where $C_1$ and $C_2$ are positive constants. Hence, we obtain
\begin{align*}
&\frac{d}{dt} \frac{c_a}{2} \int^1_{1-\delta} \left|\tilde{u}_n - \tilde{u}\right|^2 dx +  k_a \int^1_{1-\delta}  \left|\tilde{u}_{nx} - \tilde{u}_{x}\right|^2 dx\\
&\leq  k_a (\tilde{u}_{nx}(\cdot, 1) - \tilde{u}_{x}(\cdot, 1))(\tilde{u}_n(\cdot, 1) - \tilde{u}(\cdot, 1)) + \frac{c_a}{2} \int^1_{1-\delta}  \left|\tilde{u}_{n} - \tilde{u}\right|^2 dx\\
&\quad \, + C_1 \int^1_{1-\delta} \left|u_n - u\right|^2 dx + C_2 \int^1_{1-\delta} \left|u_{nx} - u_x\right|^2 dx\\
&=-(g_+(\cdot, \tilde{u}_n(\cdot, 1)) - g_+(\cdot, \tilde{u}(\cdot, 1)))(\tilde{u}_n(\cdot, 1) - \tilde{u}(\cdot, 1)) \\
&\quad \, + \frac{c_a}{2} \int^1_{1-\delta}  \left|\tilde{u}_{n} - \tilde{u}\right|^2 dx + C_1 \int^1_{1-\delta} \left|u_n - u\right|^2 dx + C_2 \int^1_{1-\delta} \left|u_{nx} - u_x\right|^2 dx\\
&\leq  \frac{c_a}{2} \int^1_{1-\delta}  \left|\tilde{u}_{n} - \tilde{u}\right|^2 dx + C_1 \int^1_{1-\delta} \left|u_n - u\right|^2 dx + C_2 \int^1_{1-\delta} \left|u_{nx} - u_x\right|^2 dx\\
&\hspace{80mm} \text{ a.e. on } [0,T].
\end{align*}
In the last inequality we use the monotonicity of $g_+$. By applying Gronwall's inequality, we have
\begin{align*}
& \frac{c_a}{2} \int^1_{1-\delta} \left|\tilde{u}_n(t,x) - \tilde{u}(t,x)\right|^2 dx +  k_a\int^t_0 \int^1_{1-\delta}  \left|\tilde{u}_{nx}(\tau, x) - \tilde{u}_{x}(\tau,x)\right|^2 dxd\tau\\
&\leq C_3 \left( \int^1_{1-\delta} \left|u_{0n} - u_0\right|^2 dx + \int^t_0  \int^1_{1-\delta}  \left|u_n(\tau, x) - u(\tau, x)\right|^2 dxd\tau \right) \text{ for } 0 \leq t \leq T,
\end{align*}
where $C_3$ is a positive constant independent of $n$.
By Lemma \ref{lem1}, it follows that
\begin{align*}
&\left|\tilde{u}_n - \tilde{u}\right|_{L^2(0,T; W^{1,2}(1-\delta, 1))} \to 0 \text{ as } n \to \infty.
\end{align*}
Hence, we have
\begin{align*}
&\left|\tilde{u}_n(\cdot,1) - \tilde{u}(\cdot,1)\right|_{L^2(0,T)}^2\\
&= \int^T_0 \left| \int^1_{1-\delta} \frac{\partial}{\partial x}\left(x(\tilde{u}_n(t,x) - \tilde{u}(t,x))\right)dx \right|^2 dt\\
&\leq 2 \left|\tilde{u}_n - \tilde{u}\right|_{L^2(0,T; W^{1,2}(1-\delta, 1))}^2 \to 0 \text{ as } n \to \infty.
\end{align*}
This shows that Lemma \ref{lem2} holds.
\end{proof}
The next lemma is a key in the proof of Theorem \ref{th1}, since it guarantees the uniform estimate for the derivative of solutions on the boundary $x=1$. In order to obtain the estimate we extend $\tilde{u}_n$ by
\begin{align}
&\tilde{u}_n(t) = \tilde{u}_{0n} + t \frac{k_a}{c_a}(\tilde{u}_{0nxx} - \zeta_{xx}u_{0n} - 2\zeta_xu_{0nx}) \text{ for } t < 0 \text{ and } n.\label{star}
\end{align}
By Lemma \ref{lem1}, it is clear that $\tilde{u}_{nt} \in L^2(-1,T;L^2(1-\delta, 1))$ for each $n$.
\begin{lemma}\label{lem4}
 Assume \Aref{4} and \Aref{5}, let $u$ be a solution of AP1($e, u_0$) for $e \in (\delta, e_0, M, T)$, where $\delta > 0$, $\delta \leq e_0 \leq 1-\delta$, $M > 0$ and $T > 0$, and put $\tilde{u} = \zeta u$. Then, there exists a positive constant $C'$ depending only on $\delta > 0$, $\left|u_b\right|_{W^{1,2}(0,T)}$, $\left|u_{0xx}\right|_{L^{2}(1-\delta,1)}$ and $\left|u_0\right|_{W^{1,2}(0,1)}$ such that
\begin{equation*}
\left|\tilde{u}_t(t)\right|_{L^2(1-\delta,1)}^2 + \int^T_0 \left|\tilde{u}_t(t)\right|_{W^{1,2}(1-\delta,1)} dt \leq C' \text{ for } 0 \leq t \leq T.
\end{equation*}
\end{lemma}
\begin{proof}
Let $\Delta t > 0$. In case $0 < t < \Delta t$, we see that
\begin{align*}
&\int^1_{1-\delta} c_a \frac{\tilde{u}_{n t}(t) - \tilde{u}_{n t}(t - \Delta t)}{\Delta t} \frac{\tilde{u}_{n }(t) - \tilde{u}_{n }(t - \Delta t)}{\Delta t}dx\\
&= \int^1_{1-\delta} k_a \frac{\tilde{u}_{n xx}(t) - \tilde{u}_{0 n xx}}{\Delta t} \frac{\tilde{u}_{n }(t) - \tilde{u}_{n }(t - \Delta t)}{\Delta t}dx\\
&\quad \, - \int^1_{1-\delta} k_a \frac{\zeta_{xx} u_{n }(t) - \zeta_{xx} u_{0 n }}{\Delta t} \frac{\tilde{u}_{n }(t) - \tilde{u}_{n }(t - \Delta t)}{\Delta t}dx\\
&\quad \, -2\int^1_{1-\delta} k_a \frac{\zeta_x u_{n x}(t) -\zeta_x u_{0 n x}}{\Delta t} \frac{\tilde{u}_{n }(t) - \tilde{u}_{n }(t - \Delta t)}{\Delta t}dx\\ 
&=: I_1(t) + I_2(t) + I_3(t) \text{ for a.e. } t \in (0, \Delta t).
\end{align*}
Easily, by integration by parts we observe that
\begin{align*}
I_1(t) &=-  \int^1_{1-\delta} k_a \frac{\tilde{u}_{n x}(t) - \tilde{u}_{0 n x}}{\Delta t} \frac{\tilde{u}_{n x}(t) - \tilde{u}_{n x}(t - \Delta t)}{\Delta t}dx\\
&\quad \,  + k_a \frac{\tilde{u}_{n x}(t,1) - \tilde{u}_{0 n x}(1)}{\Delta t} \frac{\tilde{u}_{n }(t,1) - \tilde{u}_{n }(t - \Delta t,1)}{\Delta t}\\
&=: I_{11}(t) + I_{12}(t) \text{ for a.e. } t \in (0, \Delta t).
\end{align*}
By the definition \eqref{star} of $\tilde{u}_n$ and Young's inequality, we have
\begin{align*}
I_{11}(t)
&\leq -\frac{k_a}{2}  \int^1_{1-\delta}\left|  \frac{\tilde{u}_{n x}(t) - \tilde{u}_{n x}(t - \Delta t)}{\Delta t}\right|^2 dx\\
&\quad \, + \frac{k_a^3}{2 c_a^2} \frac{\left|t-\Delta t\right|^2}{(\Delta t)^2} \int^1_{1-\delta} \left| \tilde{u}_{0nxxx} - (\zeta_{xx}u_{0n})_x - 2(\zeta_xu_{0nx})_x \right|^2 dx\\
&\hspace{80mm} \text{ for a.e. } t \in (0, \Delta t).
\end{align*}
Next, we estimate $I_{12}$ with respect to $n$ and $\Delta t$. Clearly, by \Aref{5} and the definition of $g_+$, we have
\begin{align*}
I_{12}(t)&= -\frac{1}{\Delta t}\left(g_+(t, \tilde{u}_n(t,1)) - g_+(0, \tilde{u}_{0n}(1))\right)\frac{\tilde{u}_{n }(t,1) - \tilde{u}_{n }(t - \Delta t,1)}{\Delta t}\\
&= - \frac{h}{\Delta t}\left([ \tilde{u}_n(t,1) + \theta_c]^+ - [\tilde{u}_{0n}(1) + \theta_c]^+\right)\frac{\tilde{u}_{n }(t,1) - \tilde{u}_{n }(t - \Delta t,1)}{\Delta t}\\
&\quad \, - \frac{\sigma}{\Delta t}\left(([ \tilde{u}_n(t,1) + \theta_c]^+)^4 - ([\tilde{u}_{0n}(1) + \theta_c]^+)^4\right)\frac{\tilde{u}_{n }(t,1) - \tilde{u}_{n }(t - \Delta t,1)}{\Delta t}\\
&\quad \, - \left \{\frac{h}{\Delta t}\left(u_b(t) - u_b(0)\right)  + \frac{\sigma}{\Delta t}\left(u_b(t)^4 - u_b(0)^4\right) \right \}\frac{\tilde{u}_{n }(t,1) - \tilde{u}_{n }(t - \Delta t,1)}{\Delta t}\\
&=: I_{12}^{(1)}(t) + I_{12}^{(2)}(t) +  I_{12}^{(3)}(t) \text{ for a.e. } t \in (0, \Delta t),
\end{align*}
where $h$ and $\sigma$ are positive constants in the boundary condition. For the first term, thanks to \[\tilde{u}_{0n}(1) = \tilde{u}_n(t - \Delta t, 1) - (t - \Delta t) \frac{k_a}{c_a}\tilde{u}_{0nxx}(1)\] for $t \in (0, \Delta t)$, the monotonicity and Lipschitz continuity of the positive part imply that
 \begin{align*}
  I_{12}^{(1)}(t) &=   - \frac{h}{\Delta t}([ \tilde{u}_n(t,1) + \theta_c]^+ - [\tilde{u}_{n}(t - \Delta t, 1) + \theta_c]^+)\frac{\tilde{u}_{n }(t,1) - \tilde{u}_{n }(t - \Delta t,1)}{\Delta t}\\
  &\quad \,   - \frac{h}{\Delta t}([ \tilde{u}_n(t - \Delta t,1) + \theta_c]^+ - [\tilde{u}_{0n}(1) + \theta_c]^+)\frac{\tilde{u}_{n }(t,1) - \tilde{u}_{n }(t - \Delta t,1)}{\Delta t}\\
  &\leq  \frac{h}{\Delta t}|\tilde{u}_n(t - \Delta t,1) - \tilde{u}_{0n}(1)| \left| \frac{\tilde{u}_{n }(t,1) - \tilde{u}_{n }(t - \Delta t,1)}{\Delta t} \right|\\
  &\leq  \frac{C_0 k_a h |t - \Delta t|}{c_a \Delta t}|\tilde{u}_{0nxx}(1)| \left( \left| \frac{\tilde{u}_{n }(t) - \tilde{u}_{n }(t - \Delta t)}{\Delta t}\right|_{L^2(1-\delta,1)} \right.\\
  &\quad \, \left. + \left| \frac{\tilde{u}_{nx }(t) - \tilde{u}_{nx }(t - \Delta t)}{\Delta t}\right|_{L^2(1-\delta,1)}^{1/2} \left| \frac{\tilde{u}_{n }(t) - \tilde{u}_{n }(t - \Delta t)}{\Delta t}\right|_{L^2(1-\delta,1)}^{1/2} \right)\\
&\hspace{80mm} \text{ for a.e. } t \in (0, \Delta t).
 \end{align*}
 We note that Lemma \ref{lemgag} is applied in the last inequality. According to Young's inequality, it yields that
  \begin{align*}
  I_{12}^{(1)}(t)&\leq \frac{k_a}{16} \left| \frac{\tilde{u}_{nx }(t) - \tilde{u}_{nx }(t - \Delta t)}{\Delta t}\right|_{L^2(1-\delta,1)}^2\\
& \quad \, + C_1 \left( \left| \frac{\tilde{u}_{n }(t) - \tilde{u}_{n }(t - \Delta t)}{\Delta t}\right|_{L^2(1-\delta,1)}^2 + \left|\tilde{u}_{0nxx}(1)\right|^2 \right) \text{ for a.e. } t \in (0, \Delta t),
 \end{align*} 
 where $C_1$ is a positive constant independent of $n$ and $\Delta t$. Similarly, on account  of Lemma~\ref{lem3}, we observe that
   \begin{samepage}
   \begin{align*}
  I_{12}^{(2)}(t) &=  - \frac{\sigma}{\Delta t}\left(\left([ \tilde{u}_n(t,1) + \theta_c]^+\right)^4 - \left([\tilde{u}_{n}(t - \Delta t, 1)+ \theta_c]^+\right)^4\right)\frac{\tilde{u}_{n }(t,1) - \tilde{u}_{n }(t - \Delta t,1)}{\Delta t}\\
  &\quad \,  - \frac{\sigma}{\Delta t}\left(\left([ \tilde{u}_{n}(t - \Delta t, 1) + \theta_c]^+\right)^4 - \left([\tilde{u}_{0n}(1) + \theta_c]^+\right)^4\right)\frac{\tilde{u}_{n }(t,1) - \tilde{u}_{n }(t - \Delta t,1)}{\Delta t}\\
  &\leq  \frac{\sigma}{\Delta t}\left|\tilde{u}_{n }(t - \Delta t,1) - \tilde{u}_{0n}(1)\right|\left(\left|\tilde{u}_{n }(t - \Delta t,1) + \theta_c\right| + \left| \tilde{u}_{0n}(1) + \theta_c\right|\right)\\
  & \quad \, \left( \left|\tilde{u}_{n }(t - \Delta t,1) + \theta_c\right|^2 + \left| \tilde{u}_{0n}(1) + \theta_c\right|^2\right) \left| \frac{\tilde{u}_{n }(t,1) - \tilde{u}_{n }(t - \Delta t,1)}{\Delta t} \right|\\
  &\leq \frac{k_a}{16} \left| \frac{\tilde{u}_{nx }(t) - \tilde{u}_{nx }(t - \Delta t)}{\Delta t}\right|_{L^2\left(1-\delta,1\right)}^2\\
& \quad \, + C_2 \left( \left| \frac{\tilde{u}_{n }(t) - \tilde{u}_{n }(t - \Delta t)}{\Delta t}\right|_{L^2\left(1-\delta,1\right)}^2 + \left|\tilde{u}_{0nxx}(1)\right|^2 \right) \text{ for a.e. } t \in \left(0, \Delta t\right),
   \end{align*}
\end{samepage}%
where $C_2$ is a positive constant independent of $n$ and $\Delta t$. Here, the last inequality holds, based on the boundedness of $\tilde{u}_{0n}(1)$ guaranteed by Lemma \ref{lem1}. On $I_{12}^{(3)}$, by putting $u_b(t) = u_b(0)$ for $t < 0$ and applying Young's inequality, we have
    \begin{align*}
  I_{12}^{(3)}(t)
   &\leq \frac{h}{2} \left| \frac{u_b(t) -  u_b(t - \Delta t)}{\Delta t} \right|^2 +  \frac{h}{2}\left| \frac{\tilde{u}_{n }(t,1) - \tilde{u}_{n }(t - \Delta t,1)}{\Delta t} \right|^2\\
  &\quad \, + 4 \sigma M_b^3 \left| \frac{u_b(t) -  u_b(t - \Delta t)}{\Delta t}\right| \left| \frac{\tilde{u}_{n }(t,1) - \tilde{u}_{n }(t - \Delta t,1)}{\Delta t}\right|\\
  &\leq \frac{h}{2} \left| \frac{u_b(t) -  u_b(t - \Delta t)}{\Delta t} \right|^2 +  \frac{h}{2}\left| \frac{\tilde{u}_{n }(t,1) - \tilde{u}_{n }(t - \Delta t,1)}{\Delta t} \right|^2\\
 &\quad \, + 8  \sigma^2 M_b^6  \left| \frac{u_b(t) -  u_b(t - \Delta t)}{\Delta t}\right|^2 + \frac{1}{2}\left| \frac{\tilde{u}_{n }(t,1) - \tilde{u}_{n }(t - \Delta t,1)}{\Delta t}\right|^2\\
&\hspace{80mm} \text{ for a.e. } t \in (0, \Delta t),
  \end{align*}
  where $M_b = \max_{0 \leq t \leq T}|u_b(t)|$.
   Thus, we get
      \begin{align*}
 I_1(t)&\leq  -\frac{3k_a}{8} \left| \frac{\tilde{u}_{nx }(t) - \tilde{u}_{nx }(t - \Delta t)}{\Delta t}\right|_{L^2(1-\delta,1)}^2 +C_3  \left| \frac{u_b(t) -  u_b(t - \Delta t)}{\Delta t}\right|^2\\
 &\quad \, + C_3(\left|\tilde{u}_{0n}\right|^2_{W^{1,3}(1-\delta,1)} + \left|(\zeta_{xx}u_{0n})_x\right|_{L^2(1-\delta,1)}^2 + \left|(\zeta_{x}u_{0n})_x\right|_{L^2(1-\delta,1)}^2)\\
  &\quad \, + C_3 \left| \frac{\tilde{u}_{n }(t) - \tilde{u}_{n }(t - \Delta t)}{\Delta t}\right|_{L^2(1-\delta,1)}^2 + C_3 \left| \frac{\tilde{u}_{n }(t,1) - \tilde{u}_{n }(t - \Delta t,1)}{\Delta t}\right|^2\\
  &\hspace{80mm} \text{ for a.e. } t \in (0, \Delta t),
 \end{align*} 
    where $C_3$ is a positive constant independent of $n$ and $\Delta t$.
 \par As a next step, we deal with $I_2$ and $I_3$. First, by Lemma \ref{lem3}, we have $u_n(t) = u_{0n}$ for $t \leq 0$. It is obvious that
 \begin{align*}
I_2(t) &\leq k_a \left|\zeta_{xx}\right|_{L^{\infty}(1-\delta,1)}  \int^1_{1-\delta} \left| \frac{u_{n }(t) -  u_{0 n }}{\Delta t} \right| \left| \frac{\tilde{u}_{n }(t) - \tilde{u}_{n }(t - \Delta t)}{\Delta t} \right|dx\\
 &\leq C_4 \int^1_{1-\delta}\left| \frac{u_{n }(t) - u_{n }(t - \Delta t)  }{\Delta t}\right|^2 dx +C_4 \int^1_{1-\delta} \left| \frac{\tilde{u}_{n }(t) - \tilde{u}_{n }(t - \Delta t)}{\Delta t}\right|^2 dx\\
&\hspace{80mm} \text{ for a.e. } t \in (0, \Delta t),
\end{align*}
where $C_4 =  \frac{k_a}{2} \left|\zeta_{xx}\right|_{L^{\infty}(1-\delta,1)} $. Furthermore, thanks to $\zeta_x(1) = \zeta_x(1 - \delta) = 0$ and Young's inequality, integration by parts implies that
\begin{align*}
I_3(t) &= 2 k_a \int^1_{1-\delta} \zeta_{xx}  \frac{ u_{n }(t) - u_{0 n }}{\Delta t} \frac{\tilde{u}_{n }(t) - \tilde{u}_{n }(t - \Delta t)}{\Delta t}dx\\
& \quad \, + 2 k_a \int^1_{1-\delta} \zeta_{x}  \frac{ u_{n }(t) - u_{0 n }}{\Delta t} \frac{\tilde{u}_{nx }(t) - \tilde{u}_{nx }(t - \Delta t)}{\Delta t}dx\\
&\leq  C_5 \int^1_{1-\delta}\left|  \frac{ u_{n }(t) - u_{n }(t - \Delta t)}{\Delta t}\right|^2 dx + C_5  \int^1_{1-\delta}\left|  \frac{\tilde{u}_{n }(t) - \tilde{u}_{n }(t - \Delta t)}{\Delta t} \right|^2dx\\
& \quad \,  + \frac{k_a}{4}\int^1_{1-\delta} \left|  \frac{\tilde{u}_{nx }(t) - \tilde{u}_{nx }(t - \Delta t)}{\Delta t} \right|^2 dx \text{ for a.e. } t \in (0, \Delta t),
 \end{align*}
where  $C_5$ is a positive constant independent of $n$ and $\Delta t$. Consequently, we have
\begin{align} 
&\frac{d}{dt} \frac{c_a}{2} \int^1_{1-\delta} \left| \frac{\tilde{u}_{n }(t) - \tilde{u}_{n }(t - \Delta t)}{\Delta t}\right|^2 dx +  \frac{k_a}{4}  \int^1_{1-\delta}\left|  \frac{\tilde{u}_{n x}(t) - \tilde{u}_{n x}(t - \Delta t)}{\Delta t}\right|^2 dx\notag\\
&\leq C_6 \left(  \left| \tilde{u}_{0nxxx}\right|_{L^2(1-\delta,1)}^2 + \left|(\zeta_{xx}u_{0n})_x\right|_{L^2(1-\delta,1)}^2 + \left|(\zeta_xu_{0nx})_x \right|_{L^2(1-\delta,1)}^2 \right) \notag\\
&\quad \, +C_6\left|\tilde{u}_{0nxx}(1)\right|^2 +C_6 \left| \frac{u_b(t) - u_b(t - \Delta t)}{\Delta t} \right|^2 +C_6 \int^1_{1-\delta}\left| \frac{u_{n }(t) - u_{n }(t - \Delta t)  }{\Delta t}\right|^2 dx\notag \\
& \quad \,  +C_6  \int^1_{1-\delta}\left|  \frac{\tilde{u}_{n }(t) - \tilde{u}_{n }(t - \Delta t)}{\Delta t} \right|^2dx \text{ for a.e. } t \in (0, \Delta t), \label{minus}
\end{align}
where  $C_6$ is a positive constant. 
\par Similarly to above, for $\Delta t < t < T$, we can get
\begin{align} 
&\frac{d}{dt} \frac{c_a}{2} \int^1_{1-\delta} \left| \frac{\tilde{u}_{n }(t) - \tilde{u}_{n }(t - \Delta t)}{\Delta t}\right|^2 dx +  \frac{k_a}{2}  \int^1_{1-\delta}\left|  \frac{\tilde{u}_{n x}(t) - \tilde{u}_{n x}(t - \Delta t)}{\Delta t}\right|^2 dx\notag\\
&\leq C_7 \left| \frac{u_b(t) - u_b(t - \Delta t)}{\Delta t} \right|^2
 +C_7 \int^1_{1-\delta}\left| \frac{u_{n }(t) - u_{n }(t - \Delta t)  }{\Delta t}\right|^2 dx\notag \\
& \quad \,  +C_7 \int^1_{1-\delta}\left|  \frac{\tilde{u}_{n }(t) - \tilde{u}_{n }(t - \Delta t)}{\Delta t} \right|^2dx \text{ for a.e. } t \in (\Delta t, T), \label{plus}
\end{align}
where $C_7$ is a positive constant. Indeed, we have
\begin{align*}
&\frac{d}{dt} \frac{c_a}{2} \int^1_{1-\delta} \left| \frac{\tilde{u}_{n }(t) - \tilde{u}_{n }(t - \Delta t)}{\Delta t}\right|^2 dx\\
&= \int^1_{1-\delta} k_a \frac{\tilde{u}_{n xx}(t) - \tilde{u}_{ n xx}(t - \Delta t)}{\Delta t} \frac{\tilde{u}_{n }(t) - \tilde{u}_{n }(t - \Delta t)}{\Delta t}dx\\
&\quad \, - \int^1_{1-\delta} k_a \frac{\zeta_{xx} u_{n }(t) - \zeta_{xx} u_{n }(t - \Delta t)}{\Delta t} \frac{\tilde{u}_{n }(t) - \tilde{u}_{n }(t - \Delta t)}{\Delta t}dx\\
&\quad \, -2\int^1_{1-\delta} k_a \frac{\zeta_x u_{n x}(t) -\zeta_x u_{ n x}(t - \Delta t)}{\Delta t} \frac{\tilde{u}_{n }(t) - \tilde{u}_{n }(t - \Delta t)}{\Delta t}dx\\ 
&=: I_4(t) + I_5(t) + I_6(t) \text{ for a.e. } t \in (\Delta t, T).
\end{align*}
By integration by parts and \eqref{third}, we see that
\begin{align*}
I_4(t)& = k_a \frac{\tilde{u}_{n x}(t,1) - \tilde{u}_{ n x}(t - \Delta t,1)}{\Delta t} \frac{\tilde{u}_{n }(t,1) - \tilde{u}_{n }(t - \Delta t,1)}{\Delta t}\\
& \quad \, - k_a \int^1_{1-\delta}  \left| \frac{\tilde{u}_{n x}(t) - \tilde{u}_{ n x}(t - \Delta t)}{\Delta t} \right|^2 dx\\
& =: I_{41}(t)  - k_a \int^1_{1-\delta}  \left| \frac{\tilde{u}_{n x}(t) - \tilde{u}_{ n x}(t - \Delta t)}{\Delta t} \right|^2 dx \text{ for a.e. } t \in (\Delta t, T).
\end{align*}
We also have
\begin{align*}
I_{41}(t) &= - \frac{h}{\Delta t}([ \tilde{u}_n(t,1) + \theta_c]^+  - [\tilde{u}_{n}(t - \Delta t, 1) + \theta_c]^+ )\frac{\tilde{u}_{n }(t,1) - \tilde{u}_{n }(t - \Delta t,1)}{\Delta t}\\
&\quad \,+ \frac{h}{\Delta t} (u_b(t) -  u_b(t - \Delta t))\frac{\tilde{u}_{n }(t,1) - \tilde{u}_{n }(t - \Delta t,1)}{\Delta t}\\
&\quad \, - \frac{\sigma}{\Delta t}(([ \tilde{u}_n(t,1) + \theta_c]^+)^4 - ([\tilde{u}_{n}(t - \Delta t,1 ) + \theta_c]^+)^4)\frac{\tilde{u}_{n }(t,1) - \tilde{u}_{n }(t - \Delta t,1)}{\Delta t}\\
&\quad \, + \frac{\sigma}{\Delta t} (u_b(t)^4 -  u_b(t - \Delta t)^4)\frac{\tilde{u}_{n }(t,1) - \tilde{u}_{n }(t - \Delta t,1)}{\Delta t}\\
&=: I_{41}^{(1)}(t) + I_{41}^{(2)}(t) + I_{41}^{(3)}(t) + I_{41}^{(4)}(t) \text{ for a.e. } t \in (\Delta t, T).
\end{align*}
Thanks to the Lipschitz continuity of the positive part, it follows that
\begin{align*}
 I_{41}^{(1)}(t) &\leq h \left| \frac{\tilde{u}_{n }(t,1) - \tilde{u}_{n }(t - \Delta t,1)}{\Delta t} \right|^2,
 \end{align*}
 and
\begin{samepage}
\begin{samepage}
\begin{align*} 
I_{41}^{(3)}(t) &\leq   \frac{\sigma}{\Delta t}(( \tilde{u}_n(t,1) + \theta_c)^4 - (\tilde{u}_{n}(t - \Delta t,1 ) + \theta_c)^4)\frac{\tilde{u}_{n }(t,1) - \tilde{u}_{n }(t - \Delta t,1)}{\Delta t}\\
& \leq 4 \sigma (C + \theta_c)^2 \left| \frac{\tilde{u}_{n }(t,1) - \tilde{u}_{n }(t - \Delta t,1)}{\Delta t} \right|^2 \text{ for a.e. } t \in (\Delta t, T),
 \end{align*}
\end{samepage}
\end{samepage}
 where $C$ is the positive constant defined by Lemma  \ref{lem3}. Also, by applying Young's inequality, we have
 \begin{align*}
 I_{41}^{(2)}(t) &\leq \frac{h}{2} \left| \frac{u_b(t) -  u_b(t - \Delta t)}{\Delta t} \right|^2 +  \frac{h}{2}\left| \frac{\tilde{u}_{n }(t,1) - \tilde{u}_{n }(t - \Delta t,1)}{\Delta t} \right|^2,
  \end{align*}
and
 \begin{align*}
 I_{41}^{(4)}(t) &\leq 4 \sigma M_b^3 \left| \frac{u_b(t) -  u_b(t - \Delta t)}{\Delta t}\right| \left| \frac{\tilde{u}_{n }(t,1) - \tilde{u}_{n }(t - \Delta t,1)}{\Delta t}\right|\\
 &\leq 8  \sigma^2 M_b^6  \left| \frac{u_b(t) -  u_b(t - \Delta t)}{\Delta t}\right|^2 + \frac{1}{2}\left| \frac{\tilde{u}_{n }(t,1) - \tilde{u}_{n }(t - \Delta t,1)}{\Delta t}\right|^2\\
&\hspace{80mm} \text{ for a.e. } t \in (\Delta t, T),
  \end{align*}
  where $M_b = \max_{0 \leq t \leq T}|u_b(t)|$. Hence, we obtain
   \begin{align*}
   I_4(t) &\leq C_8  \left| \frac{u_b(t) -  u_b(t - \Delta t)}{\Delta t}\right|^2 + C_8 \left| \frac{\tilde{u}_{n }(t,1) - \tilde{u}_{n }(t - \Delta t,1)}{\Delta t}\right|^2\\
   &\quad \,  - k_a \int^1_{1-\delta}  \left| \frac{\tilde{u}_{n x}(t) - \tilde{u}_{ n x}(t - \Delta t)}{\Delta t} \right|^2 dx \text{ for a.e. } t \in (\Delta t, T),
 \end{align*}   
 where $C_8$ is a positive constant. Similarly to $I_2$ and $I_3$, we have
  \begin{align*}
I_5(t) 
 &\leq C_9 \int^1_{1-\delta}\left| \frac{u_{n }(t) - u_{n }(t - \Delta t)  }{\Delta t}\right|^2 dx +C_9 \int^1_{1-\delta} \left| \frac{\tilde{u}_{n }(t) - \tilde{u}_{n }(t - \Delta t)}{\Delta t}\right|^2 dx\\
&\hspace{70mm} \text{ for a.e. } t \in  (\Delta t, T),
\end{align*}
and
\begin{align*}
I_6(t) 
&\leq  C_{10} \int^1_{1-\delta}\left|  \frac{ u_{n }(t) - u_{n }(t - \Delta t)}{\Delta t}\right|^2 dx + C_{10}  \int^1_{1-\delta}\left|  \frac{\tilde{u}_{n }(t) - \tilde{u}_{n }(t - \Delta t)}{\Delta t} \right|^2dx\\
& \quad \,  + \frac{k_a}{2}\int^1_{1-\delta} \left|  \frac{\tilde{u}_{nx }(t) - \tilde{u}_{nx }(t - \Delta t)}{\Delta t} \right|^2 dx \text{ for a.e. } t \in  (\Delta t, T),
 \end{align*}
where $C_9 $ and $C_{10}$ are positive constants. Accordingly, we have
 \begin{align*}
 &\frac{d}{dt} \frac{c_a}{2} \int^1_{1-\delta} \left| \frac{\tilde{u}_{n }(t) - \tilde{u}_{n }(t - \Delta t)}{\Delta t}\right|^2 dx  + \frac{k_a}{2}\int^1_{1-\delta} \left|  \frac{\tilde{u}_{nx }(t) - \tilde{u}_{nx }(t - \Delta t)}{\Delta t} \right|^2 dx \\
  &\leq C_{11}  \left| \frac{u_b(t) -  u_b(t - \Delta t)}{\Delta t}\right|^2 + C_{11} \left| \frac{\tilde{u}_{n }(t,1) - \tilde{u}_{n }(t - \Delta t,1)}{\Delta t}\right|^2\\
  &\quad \, + C_{11} \int^1_{1-\delta}\left|  \frac{ u_{n }(t) - u_{n }(t - \Delta t)}{\Delta t}\right|^2 dx + C_{11}  \int^1_{1-\delta}\left|  \frac{\tilde{u}_{n }(t) - \tilde{u}_{n }(t - \Delta t)}{\Delta t} \right|^2dx\\
  &\hspace{80mm} \text{ for a.e. } t \in  (\Delta t, T),
  \end{align*}
  where $C_{11}$ is a positive constant. Due to Lemma \ref{lemgag} and Young's inequality, \eqref{plus} holds.
\par From \eqref{minus} and \eqref{plus}, there exists a positive constant $C_{12}$ independent of $n$ and $\Delta t$ such that
\begin{align} 
&\frac{d}{dt} \frac{c_a}{2} \int^1_{1-\delta} \left| \frac{\tilde{u}_{n }(t) - \tilde{u}_{n }(t - \Delta t)}{\Delta t}\right|^2 dx +  \frac{k_a}{4}  \int^1_{1-\delta}\left|  \frac{\tilde{u}_{n x}(t) - \tilde{u}_{n x}(t - \Delta t)}{\Delta t}\right|^2 dx\notag\\
&\leq C_{12} \chi_{\Delta t}(t) \alpha_{0n}  +C_{12} \left| \frac{u_b(t) - u_b(t - \Delta t)}{\Delta t} \right|^2 +C_{12} \int^1_{1-\delta}\left| \frac{u_{n }(t) - u_{n }(t - \Delta t)  }{\Delta t}\right|^2 dx\notag\\
& \quad \,  +C_{12} \int^1_{1-\delta}\left|  \frac{\tilde{u}_{n }(t) - \tilde{u}_{n }(t - \Delta t)}{\Delta t} \right|^2dx \text{ for a.e. } t \in [0, T], \label{ine}
\end{align}
where
\begin{align*}
 \chi_{\Delta t}(t) =
\begin{cases}
&1 \text{ for } 0 < t < \Delta t,\\
&0 \text{ for }  \Delta t < t < T,
\end{cases}
\end{align*}
and
\begin{align*}
\alpha_{0n} =\left|\tilde{u}_{0nxx}(1)\right|^2 +  \left| \tilde{u}_{0nxxx}\right|_{L^2(1-\delta,1)}^2 + \left|(\zeta_{xx}u_{0n})_x\right|_{L^2(1-\delta,1)}^2 + \left|(\zeta_xu_{0nx})_x \right|_{L^2(1-\delta,1)}^2.
\end{align*}
By applying Gronwall's inequality to \eqref{ine}, we see that
\begin{align*} 
& \frac{c_a}{2} \int^1_{1-\delta} \left| \frac{\tilde{u}_{n }(t_1) - \tilde{u}_{n }(t_1 - \Delta t)}{\Delta t}\right|^2 dx +  \frac{k_a}{4} \int^{t_1}_0\!\!\!  \int^1_{1-\delta}\left|  \frac{\tilde{u}_{n x}(t) - \tilde{u}_{n x}(t - \Delta t)}{\Delta t}\right|^2 dxdt\\
&\leq \frac{c_a}{2} \int^1_{1-\delta} \left| \frac{\tilde{u}_{n }(0) - \tilde{u}_{n }( - \Delta t)}{\Delta t}\right|^2 dx  + C_{13} \alpha_{0n} \int^{t_1}_0 \chi_{\Delta t}(t) dt\\
& \quad \,   +C_{13} \int^{t_1}_0 \left| \frac{u_b(t) - u_b(t - \Delta t)}{\Delta t} \right|^2 dt\\
&\quad \,  +C_{13} \int^{t_1}_0\!\!\!\int^1_{1-\delta}\left| \frac{u_{n }(t) - u_{n }(t - \Delta t)  }{\Delta t}\right|^2 dxdt \text{ for } t_1 \in [0, T],
\end{align*}
where $C_{13}$ is a positive constant independent of $n$ and $\Delta t$. \\Because of $\tilde{u}_{n }(0) - \tilde{u}_{n }( - \Delta t) = \Delta t \frac{k_a}{c_a} \zeta u_{0nxx},$ we have
\begin{align} 
&  \frac{c_a}{2} \int^1_{1-\delta} \left| \frac{\tilde{u}_{n }(t_1) - \tilde{u}_{n }(t_1 - \Delta t)}{\Delta t}\right|^2 dx + \frac{k_a}{4} \int^{T}_0\!\!\!  \int^1_{1-\delta}\left|  \frac{\tilde{u}_{n x}(t) - \tilde{u}_{n x}(t - \Delta t)}{\Delta t}\right|^2 dxdt\notag\\
&\leq \frac{k_a^2}{2c_a} \int^1_{1-\delta} \left| \zeta u_{0nxx}\right|^2 dx  + C_{13} \alpha_{0n} \Delta t  +C_{13} \int^{T}_0 \left| \frac{u_b(t) - u_b(t - \Delta t)}{\Delta t} \right|^2 dt\notag\\
&\quad \,  +C_{13} \int^{T}_0\!\!\!\int^1_{1-\delta}\left| \frac{u_{n }(t) - u_{n }(t - \Delta t)  }{\Delta t}\right|^2 dxdt \text{ for any } n \text{ and } \Delta t.\label{toyama}
\end{align}
This shows that $\{ \frac{\tilde{u}_{n} - \tilde{u}_{n}(\cdot - \Delta t)}{\Delta t} : \Delta t > 0 \}$ and  $\{ \frac{\tilde{u}_{n x} - \tilde{u}_{n x}(\cdot - \Delta t)}{\Delta t} : \Delta t > 0 \}$ are bounded in $L^{\infty}(0,T;$ $L^2(1-\delta,1))$ and $L^2(0,T; L^2(1-\delta,1))$, respectively.  It is clear that \[\frac{\tilde{u}_{n } - \tilde{u}_{n }(\cdot - \Delta t)}{\Delta t} \to \tilde{u}_{n t} 
\text{ weakly$\ast$ in } L^{\infty}(0,T; L^2(1-\delta,1))\] and \[\frac{\tilde{u}_{n x} - \tilde{u}_{nx }(\cdot - \Delta t)}{\Delta t} \to \tilde{u}_{n xt} 
\text{ weakly in } L^2(0,T; L^2(1-\delta,1)) \text{ as } \Delta t \to 0.\] Accordingly, from \eqref{toyama} it follows that
\begin{align*} 
& \frac{c_a}{2}  \int^1_{1-\delta}\left|  \tilde{u}_{n t}(t)\right|^2 dx +  \frac{k_a}{4} \int^{t}_0\!\!\!  \int^1_{1-\delta}\left|  \tilde{u}_{n x\tau}\right|^2 dxd\tau\\
&\leq \frac{k_a^2}{2c_a} \int^1_{1-\delta} \left| \zeta u_{0nxx}\right|^2 dx  +C_{13} \int^{t}_0 \left| u'_b\right|^2 d\tau  +C_{13} \int^{t}_0\!\!\!\int^1_{1-\delta}\left| u_{n \tau}\right|^2 dxd\tau\\
&\hspace{80mm} \text{ for }t \in [0,T] \text{ and any } n.
\end{align*}
By using Lemma \ref{lem1}, we observe that $\{\tilde{u}_{nxt}\}$ is bounded in  $L^2(0,T; L^2(1-\delta,1))$. Hence, due to Lemma \ref{lem2} it is obvious  that $\tilde{u}_{nxt} \in L^2(0,T; L^2(1-\delta,1))$ and it satisfies
\begin{align*} 
& \frac{c_a}{2}  \int^1_{1-\delta}\left|  \tilde{u}_{ t}(t)\right|^2 dx +  \frac{k_a}{4} \int^{t}_0\!\!\!  \int^1_{1-\delta}\left|  \tilde{u}_{x\tau}\right|^2 dxd\tau\\
& \leq \frac{k_a^2}{2c_a} \int^1_{1-\delta} \left| \zeta u_{0xx}\right|^2 dx  +C_{13} \int^{t}_0 \left| u'_b\right|^2 d\tau  +C_{13} \int^{t}_0\!\!\!\int^1_{1-\delta}\left| u_{\tau}\right|^2 dxd\tau \text{ for } 0 \leq t \leq T.
\end{align*}
Thus, Lemma \ref{lem4} holds.
\end{proof}
\begin{proof}[Proof of Proposition \ref{prop1}]
By Lemma \ref{lem4}, there exists a positive constant $R_b$ such that 
\begin{align*} 
  \int^1_{1-\delta}\left|  \tilde{u}_{ t}(t)\right|^2 dx + \int^{t}_0\!\!\!  \int^1_{1-\delta}\left|  \tilde{u}_{x\tau}\right|^2 dxd\tau \leq R_b \text{ for } 0 \leq t \leq T.
\end{align*}
Hence, we have
\begin{align*} 
&\int^T_0 \left|\tilde{u}_t(\cdot,1)\right|^2 dt\\
&= \int^T_0 \!\!\! \int^1_{1-\delta} \frac{\partial}{\partial x}\left(\frac{1}{\delta}(x-(1-\delta))\tilde{u}_t(\cdot,x)^2\right) dxdt\\
&\leq \left( \frac{1}{\delta} + 1 \right)\int^T_0 \!\!\! \int^1_{1-\delta} \tilde{u}_t(\cdot,x)^2 dx dt + \int^T_0 \!\!\! \int^1_{1-\delta} \tilde{u}_{tx}(\cdot,x)^2 dx dt\\
&\leq \left( \frac{1}{\delta} + 1 \right) (R_b T + R_b).
\end{align*}
This implies that $\tilde{u}(\cdot, 1) \in W^{1,2}(0,T)$, namely, $u(\cdot, 1) \in W^{1,2}(0,T)$. Moreover, it holds
\begin{align*}
\left|u_t(t)\right|_{L^2(1-\delta /2,1)} &\leq \Big| \frac{1}{\zeta} \tilde{u}_t(t)\Big|_{L^2(1-\delta /2,1)}\\
&\leq \frac{1}{d_0} \left| \tilde{u}_t(t)\right|_{L^2(1-\delta /2,1)} \text{ for } 0 \leq t \leq T,
\end{align*}
where $d_0 = \inf_{x \in [1-\delta/2, 1]} \zeta(x)$. Thus, we have proved Proposition \ref{prop1}.
\end{proof}
\section{Estimates for water contents}\label{esw}
We give some uniform estimates for the water content in order to prove Theorem \ref{th1}. The proofs of Lemmas \ref{lem-w-1} and \ref{lem-w-2} are similar to Lemmas \ref{lemma100} and \ref{lemma101}, respectively. Thus, we omit the proofs.
\begin{lemma}[{cf. Proposition 3.3 in \cite{AK}}] \label{lem-w-3}
Let $T>0$, and assume that
\begin{align*}
&e \in W^{1,2}(0,T), \ \quad 0<e<1 \ \text{ on } [0,T],\\
&u_0 \in W^{1,2}(0,1), \quad u_0 \geq 0 \ \text{ on } [e(0),1], \quad u_0 \leq 0 \ \text{ on } [0,e(0)],\\
&u_b \in W^{1,2}(0,T),  \quad u_b \geq \theta_c \text{ on } [0,T],\\
&w_0 \in W^{1,2}(0,1).
\end{align*}
 Then AP($e$) has a unique solution $(u,w)$
 on $[0,T]$, where $u$ and $w$ satisfy \{\Sref{2}, \Sref{3}\} and \Sref{4}, respectively.
\end{lemma}
\begin{proof}
Since $u(\cdot,1) \in W^{1,2}(0,T)$ as proved in Proposition \ref{prop1}, by applying Lemma \ref{prop3-2}, we can show this lemma.
\end{proof}
\begin{lemma}[{cf. Lemmas 4.3 and 4.6 in \cite{AK}}] \label{lem-w-1}
Suppose the same assumption as in Lemma~\ref{lemma100} and Lemma~\ref{lem-w-3}. Let $(u, w)$ be a solution of $\text{AP}(e)$ on $[0,T]$ for $e \in K(\delta, e_0, M,T_0)$. Then, there exists a positive constant \[C_1 = C_1\left(\delta, M, T, \left|u_b\right|_{W^{1,2}(0,T)}, \left|u_{0xx}\right|_{L^{2}(1-\delta,1)}, \left|u_0\right|_{W^{1,2}(0,1)}, \left|w_0\right|_{W^{1,2}(0,1)}\right)\] such that
\begin{align*}
 \int^1_0 w(t,x)^2dx +\int^{t}_{0} \!\!\! \int^{1}_{0} \left|w_x(\tau,x)\right|^2 dx d\tau \leq C_1 \text{ for } 0 \leq t \leq T,
\end{align*}
and
\begin{align}
 \int^1_0 \left|w_x(t,x)\right|^2dx +\int^{t}_{0} \!\!\! \int^{1}_{0} \left|w_{\tau}(\tau,x)\right|^2 dx d\tau \leq C_1 \text{ for } 0 \leq t \leq T.\label{cloud}
\end{align}
\end{lemma}
\begin{remark}\label{rem-w-1}
 Because of $u(\cdot,1) \in W^{1,2}(0,T)$ by Proposition \ref{prop1}, we can prove \eqref{cloud}.
\end{remark}
\begin{lemma}[cf. Lemma 5.2 in \cite{AK}]\label{lem-w-2}
  Let \[\delta > 0, M>0, T > 0, e_{0} \in [\delta, 1-\delta]\] and $e_i \in K(\delta, e_{0}, M,T)$ for $i=1,2$ and assume \Aref{2} and \Aref{3}. If $(u_i, w_i)$ is a solution of AP($e_i$) for $i =1,2$, then, for some positive constant \[C_2 = C_2\left(\delta, M,T, \left|u_b\right|_{W^{1,2}(0,T)}, \left|u_{0xx}\right|_{L^{2}(1-\delta,1)}, \left|u_0\right|_{W^{1,2}(0,1)}, \left|w_0\right|_{W^{1,2}(0,1)}\right)\] it holds that
 \begin{align*}
&\int^1_0 \left|w _{1}(t,x) - w_{2}(t,x)\right|^2 dx +  \int^{t}_{0} \!\!\! \int^{1}_{0}  \left|w _{1x}(\tau,x) - w_{2x}(\tau,x)\right|^2dx d\tau \\
&  \leq C_{2} \int^{t}_0 \left(\left|e_{1}(\tau) - e_{2}(\tau)\right|^2 + \left|e'_{1}(\tau) - e'_{2}(\tau)\right|^2 \right) d\tau \text{ for } 0 \leq t \leq T.
\end{align*}
\end{lemma}
\section{Proof of Theorem \ref{th1}}\label{proofth1}
In this section, first, by applying Banach's fixed point theorem to a solution operator of AP($e$), we prove the existence of a solution of P locally in time. As mentioned in Section \ref{sec10}, for proving \eqref{maru}, we need $e \in W^{1,3}(0,T)$. Let $K(\delta, e_0, M,T_0)$ be a subset of $ W^{1,3}(0,T)$ given in Section \ref{sec10}, for $0 < \delta < e_0 < 1-\delta < 1$, $M > 0$ and $T > 0$.
Moreover, for $T_0 \in (0, T]$, we define a solution operator $\Gamma : K(\delta, e_0, M,T_0) \to W^{1,3}(0,T_0)$ as follows:
\begin{align}
&\Gamma(e)(t) = e_0 + \int_0^t \frac{1}{l w(\tau, e(\tau)) } \left(   k_l u_x(\tau, e(\tau)-) - k_a u_x(\tau, e(\tau)+) \right) d\tau\label{gamma}\\
&\hspace{70mm} \text{ for } 0 \leq t \leq T \text{ and } e \in K(\delta, e_0, M, T_0),\notag
\end{align}
where $(u,w)$ is a solution of AP($e$) on $[0,T]$. Clearly, the set $K(\delta, e_0, M,T_0)$ is closed in $W^{1,3}(0,T_0)$, and $\Gamma(e) \in W^{1,3}(0,T_0)$ for any $e \in K(\delta, e_0, M, T_0)$. 
\begin{lemma}\label{lem3-18}
 Let $0 < \delta < 1$ and $M > 0$. Suppose the same assumption as in the later part of Lemma \ref{lemma100} and $w_0(e_0) \geq 2\delta_1$ for some $\delta_1 > 0$, then there exists $T_1 > 0$ such that ${w}(t,e(t)) \geq \delta_1$ for any $t \in [0, T_1]$.
\end{lemma}
\begin{proof}
First, it is easy to see that
\begin{align*}
w(t,e(t)) \geq w_0(e_0) - | w(t,e(t)) - w_0(e(t))| - |w_0(e(t)) - w_0(e_0)| \text{ for } 0 \leq t \leq T.
\end{align*}
Here, we note that 
\begin{align*}
& \left| w(t,e(t)) - w_0(e(t))\right|^2\\
& =\frac{1}{e(t)} \int^{e(t)}_0 \frac{\partial}{\partial x} (x(w(t,x) - w_0(x))^2) dx\\
&\leq \frac{1}{\delta} \left( \int^1_0 \left|w(t) - w_0\right|^2 dx + 2 \int^1_0 |w_x(t) - w_{0x}| |w(t) - w_0| dx \right)\\
&\leq \frac{1}{\delta} \left( t \int^t_0 \left|w_{\tau}\right|_{L^2(0,1)}^2 d\tau + 2  t^{1/2} \left(\int^t_0 \left|w_{\tau}\right|_{L^2(0,1)}^2 d\tau\right)^{1/2}  (\left|w_x(t)\right|_{L^2(0,1)} + \left|w_{0x}\right|_{L^2(0,1)}) \right)\\
&\hspace{100mm}\text{for } 0 \leq t \leq T.
\end{align*}
Hence, thanks to the assumption, Lemmas \ref{lem-w-1} and \ref{lem-w-2} imply that
\begin{align*}
w(t,e(t))& \geq 2 \delta_1 - R_1 t^{1/2} - \left|e(t) - e_0\right|^{1/2}\left|w_{0x}\right|_{L^2(0,1)}\\
&\geq 2 \delta_1 - R_1 t^{1/2} - \left|e'\right|_{L^3(0,t)}^{1/2} t^{1/3} \left|w_{0x}\right|_{L^2(0,1)} \text{ for } 0 \leq t \leq T,
\end{align*}
where $R_1$ is a positive constant. Thus, we can prove this lemma.
\end{proof}
\begin{lemma}\label{lem3-19}
 Let $M > 0$. Suppose the same assumption as in Lemma \ref{lem3-18}, then there exist $\delta > 0$ and $T_2 \in (0, T_1]$ such that $\delta \leq \Gamma (e)(t) \leq 1-\delta \text{ for } 0 \leq t \leq T_2$ and $e \in K(\delta,e_0,M,T)$.
\end{lemma}
\begin{proof}
First, we choose $\delta>0$ such that $2\delta \leq e_0 \leq 1-2\delta$. For any $e \in K(\delta, e_0, M,T)$, by applying Lemmas \ref{lemma100} and \ref{lem3-18}, we have
\begin{align*}
  &  |\Gamma (e)(t) - e_0|\\
&  = \left|  \int_0^t \frac{1}{l w(\tau, e(\tau)) } \left(   k_l u_x(\tau, e(\tau)-) - k_a u_x(\tau, e(\tau)+) \right) d\tau \right|\\
& \leq \frac{k^* C_1}{l \delta_1 } t \text{ for } 0 \leq t \leq T_1,
 \end{align*}
 where $k^* = \max \{k_l, k_a\}$ and  $\displaystyle C_{1}$ is a positive constant defined by \eqref{maru} in Lemma \ref{lemma100}.
Hence, by taking $0 < T_2 \leq T_1$ satisfying $ \frac{k^* C_1}{l \delta_1 }  T_2 \leq \delta$, we obtain
\[\delta  \leq \Gamma (e)(t) \leq 1-\delta \text{ for }  0 \leq t \leq T_2.\qedhere\]
\end{proof}
\begin{lemma}\label{lem3-20}
 Suppose the same assumption as in Lemma \ref{lem3-19}. Then, there exists $T_3$ in $(0,T_2]$ such that $\displaystyle \int^{T_3}_0 \left|\frac{d}{dt} \Gamma(e)(t)\right|^3 dt \leq M$ for $e\in K(\delta, e_0, M,T)$.
\end{lemma}
\begin{proof}
By using Lemmas \ref{lemma100} and \ref{lem3-18}, we have
\begin{align*}
  &  \int^{t}_0 \left|\frac{d}{dt} \Gamma (e)(t)\right|^3 dt =\int^{t}_0 \left|\frac{1}{l w(\tau, e(\tau)) } \left(   k_l u_x(\tau, e(\tau)-) - k_a u_x(\tau, e(\tau)+) \right) \right|^3 d\tau \\
 & \leq \frac{1}{l^3 \delta_1^3 } \int^{t}_0 \bigg \{ \left|  k_l u_x(\tau, e(\tau)-)\right| + \left |k_a u_x(\tau, e(\tau)+) \right|\bigg \} ^3 d\tau \leq C_{1} t,
 \end{align*}
 where $\displaystyle C_{1}$ is a positive constant.
 By choosing small $T_3 \in (0, T_2]$ with $C_{1} T_3 \leq M$, we can prove this lemma.
\end{proof}
\begin{lemma}\label{lem3-26}
 Under the same assumption as in Lemma \ref{lem3-18}, there exists $\lambda \in (0,1)$ and $T_0 \in (0, T_3]$ such that 
\[\left|\Gamma(e_1)-\Gamma(e_2)\right|_{W^{1, 3}(0,T_0)} \leq \lambda \left|e_1-e_2\right|_{W^{1,3}(0,T_0)}.\]
\end{lemma}
\begin{proof}
For $0 < T_4 \leq T_3$, we have
\begin{align*}
&   \left|  \frac{d}{dt}(\Gamma(e_1)-\Gamma(e_2))\right|_{L^3(0,T_4)} \\
& = \left| \frac{1}{l w_1(\cdot, e_1) } \left(  k_l u_{1x}(\cdot, e_1-) -  k_a u_{1x}(\cdot, e_1+) \right) \right.\\
&\quad \, \left. - \frac{1}{l w_2(\cdot, e_2) } \left(   k_l u_{2x}(\cdot, e_2-) -  k_a u_{2x}(\cdot, e_2+) \right) \right|_{L^3(0,T_4)}\\
&\leq \left|\frac{k_l}{l} \left( \frac{1}{w_1(\cdot, e_1)}\left(  u_{1x}(\cdot, e_1-)\right)  - \frac{1}{w_2(\cdot, e_2)}\left(  u_{2x}(\cdot, e_2-)\right) \right) \right|_{L^3(0,T_4)}\\
&\quad \, + \left|\frac{k_a}{l} \left( \frac{1}{w_1(\cdot, e_1)}\left(  u_{1x}(\cdot, e_1+)\right)  - \frac{1}{w_2(\cdot, e_2)}\left(  u_{2x}(\cdot, e_2+)\right) \right) \right|_{L^3(0,T_4)}\\
&=: I_{1} + I_{2} 
 \end{align*}
 First, we estimate $I_{1}$. Easily, we have
\begin{align*}
 I_{1}& \leq \frac{k_l}{l}\left| \left(\frac{1}{w_1(\cdot, e_1)} - \frac{1}{w_2(\cdot, e_2)}\right)u_{1x}(\cdot, e_1-)\right|_{L^3(0,T_4)} \\
&\quad \, + \frac{k_l}{l} \left| \frac{1}{w_2(\cdot, e_2)} \left(u_{1x}(\cdot, e_1-) - u_{2x}(\cdot, e_2-)\right)\right|_{L^3(0,T_4)} \\
& =\frac{k_l}{l} (I_{11} + I_{12}).
\end{align*}
Let $\delta_1$ and $C_1$ be positive constants given in Lemma \ref{lem3-18} and \eqref{maru}, respectively. Then, we have
\begin{align*}
     I_{11} & \leq \frac{C_{1}}{\delta_1^2 } \Bigg\{ \left( \int^{T_4}_0   \left|{w}_1(t, e_1(t)) - {w}_2(t, e_1(t))\right|^3 dt \right)^{\frac{1}{3}}\\
     &\hspace{10mm} +  \left( \int^{T_4}_0   \left|{w}_2(t, e_1(t)) - {w}_2(t, e_2(t))\right|^3 dt \right)^{\frac{1}{3}} \Bigg\} =:I_{11}^{(1)} + I_{11}^{(2)}.
\end{align*}
By applying Lemmas \ref{lemgag} and \ref{lem-w-2}, we see that
\begin{align*}
\left(I_{11}^{(1)}\right)^3&\leq \Bigg (\frac{C_{1}}{\delta_1^2 }\Bigg)^3 \int^{T_4}_0   \left|{w}_1 - {w}_2\right|^3_{L^{\infty}(0,1)} dt \\
&\leq C_2 \int^{T_4}_0  \left( \left|{w}_1 - {w}_2\right|_{L^2(0,1)} +  \left|{w}_1 - {w}_2\right|_{L^2(0,1)}^{1/2} \left|{w}_{1x} - {w}_{2x}\right|_{L^2(0,1)}^{1/2}\right) dt \\
&\leq 8 C_2 \Bigg( \left|{w}_1 - {w}_2\right|_{L^{\infty}(0,T_4;L^2(0,1))}^{3} T_4\\
&\hspace{10mm} + \left|{w}_1 - {w}_2\right|_{L^{\infty}(0,T_4;L^2(0,1))}^{3/2} \int^{T_4}_0   \left|{w}_{1x} - {w}_{2x}\right|_{L^2(0,1)}^{3/2} dt\Bigg)\\
&\leq 8 C_2 \Bigg( \left|{w}_1 - {w}_2\right|_{L^{\infty}(0,T_4;L^2(0,1))}^{3} T_4\\
&\hspace{10mm} + \left|{w}_1 - {w}_2\right|_{L^{\infty}(0,T_4;L^2(0,1))}^{3/2}    \left|{w}_{1x} - {w}_{2x}\right|_{L^2(0,T_4;L^2(0,1))}^{3/2} T_4^{1/4}\Bigg)\\
&\leq C_3 T_4^{1/4} \Bigg( \int^{T_4}_0 \left(\left|e_{1} - e_{2}\right|^2 + \left|e'_{1} - e'_{2}\right|^2 \right) dt \Bigg)^{3/2}\\
&\leq C_3 T_4^{3/4} \int^{T_4}_0 \left(\left|e_{1} - e_{2}\right|^3 + \left|e'_{1} - e'_{2}\right|^3 \right) dt,
\end{align*}
and
\begin{align*}
I_{11}^{(1)} \leq C_3^{1/3} T_4^{1/4} \left(\left|e_{1} - e_{2}\right|_{L^3(0,T_4)}+ \left|e'_{1} - e'_{2}\right|_{L^3(0,T_4)} \right),
\end{align*}
where $C_2$ and $C_3$ are positive constants. Next, by applying Lemma \ref{lemgag}, again, in case $e_1(t) \leq e_2(t)$ for some $t \in [0,T]$,
\begin{align*}
&\left|{w}_2(t, e_1(t)) - {w}_2(t, e_2(t))\right|\\
&\leq \left|w_{2x}(t)\right|_{L^{\infty}(0,e_2(t))} \left|e_1(t) - e_2(t)\right|\\
&\leq C \left(\left|w_{2x}(t)\right|_{L^{2}(0,e_2(t))} + \left|w_{2x}(t)\right|_{L^{2}(0,e_2(t))}^{1/2}\left|w_{2xx}(t)\right|_{L^{2}(0,e_2(t))}^{1/2}\right) \left|e_1(t) - e_2(t)\right|\\
&\leq C \left(\left|w_{2x}(t)\right|_{L^{2}(0,1)} + \left|w_{2x}(t)\right|_{L^{2}(0,1)}^{1/2} \left(\left|w_{2xx}(t)\right|_{L^{2}(0,e_2(t))}^{1/2} + \left|w_{2xx}(t)\right|_{L^{2}(e_2(t),1)}^{1/2}\right)\right)\\
& \quad \, \times \left|e_1(t) - e_2(t)\right|.
\end{align*}
Even if $e_1(t) > e_2(t)$, this inequality still holds. Accordingly, it is clear that
\begin{samepage}
\begin{align*}
\left(I_{11}^{(2)}\right)^3&\leq C_4 \left|e_1(t) - e_2(t)\right|^3_{C([0,T_4])}  \int^{T_4}_0 \Big( \left|w_{2x}(t)\right|_{L^{2}(0,1)} \\
&\hspace{20mm}+ \left|w_{2x}(t)\right|_{L^{2}(0,1)}^{1/2}\left(\left|w_{2xx}(t)\right|_{L^{2}(0,e_2(t))}^{1/2} + \left|w_{2xx}(t)\right|_{L^{2}(e_2(t),1)}^{1/2}\right) \Big)^3 dt\\
&\leq 8 C_4 T_4^2 \left|e'_1 - e'_2\right|_{L^3(0,T_4)}^3 \Bigg\{  \left|w_{2x}\right|_{L^{\infty}(0,T_4;L^{2}(0,1))}^3 T_4\\
&\hspace{10mm}+ \left|w_{2x}\right|_{L^{\infty}(0,T_4;L^{2}(0,1))}^{3/2} \int^{T_4}_0 \left(\left|w_{2xx}(t)\right|_{L^{2}(0,e_2(t))}^{3/2} + \left|w_{2xx}(t)\right|_{L^{2}(e_2(t),1)}^{3/2}\right)dt \Bigg\} \\
&\leq  C_5 T_4^2 \left|e'_1 - e'_2\right|_{L^3(0,T_4)}^3,
\end{align*}
\end{samepage}
where $C_4$ and $C_5$ are positive constants. Namely, for some positive constant $C_6$ it holds that
\begin{align*}
I_{11} \leq C_6 T_4^{1/4}\left|e_1 - e_2\right|_{W^{1,3}(0,T_4)}.
\end{align*}
On $I_{12}$, by using \eqref{tono}, we have
\begin{align}
I_{12} &\leq  \frac{1}{\delta_1}  \left| u_{1x}(\cdot, e_1-) - u_{2x}(\cdot, e_2-)\right|_{L^3(0,T_4)}\notag \\
& =  \frac{1}{\delta_1}\left(\int^{T_4}_0 \left| u_{1x}(t, e_1(t)-) - u_{2x}(t, e_2(t)-)\right|^3 dt\right)^{\frac{1}{3}}\notag\\
 &\leq C_7 \left|e_1 - e_2\right|_{W^{1,3}(0,T_4)} T_4^{\frac{1}{12}}\label{tou},
 \end{align}
where $C_{7}$ is a positive constant.
Hence, we have
\[I_{1} \leq C_{8} T_4^{\frac{1}{12}} \left|e_1-e_2\right|_{W^{1,3}(0,T_4)} \text{ for } 0 \leq T_4 \leq T_3,\]
where $C_{8}$ is a positive constant. Similarly to $I_{1}$, we can get same estimate for $I_2$. Thus, we obtain
\[ \displaystyle \left|\frac{d}{dt}(\Gamma(e_1)-\Gamma(e_2))\right|_{L^3(0,T_4)} \leq C_{9} T_4^{\frac{1}{12}} \left|e_1-e_2\right|_{W^{1,3}(0,T_4)} \text{ for } 0 \leq T_4 \leq T_3,\]
where $C_{9}$ is a positive constant. By using above inequality, we have
 \begin{align*}
   \left|\Gamma(e_1)-\Gamma(e_2)\right|_{L^3(0,T_4)}
   & \leq T_4 \left| \frac{d}{d\tau}(\Gamma(e_1)-\Gamma(e_2))\right|_{L^3(0,T_4)} \\
   & \leq C_{10} T_4 \left|e_1-e_2\right|_{W^{1,3}(0,T_4)} \text{ for }  0 \leq T_4 \leq T_3,
 \end{align*}
 where $C_{10}$ is a positive constant. Therefore,  Lemma \ref{lem3-26} have been proved.
\end{proof}
Since Lemmas \ref{lem3-18} - \ref{lem3-26}  hold, we can choose small $T_0 \in (0,T]$ such that $\Gamma$ is a function $K(\delta, e_0, M$, $T_0)$ $\to K(\delta, e_0, M, T_0)$ and a contraction mapping on $W^{1,3}(0,T_0)$. Hence, there exists one and only one $e \in K(\delta, e_0, M, T_0)$ with $\Gamma(e) = e$ by applying  Banach's fixed-point theorem to $\Gamma$. Obviously, \eqref{FBP1} holds on $[0,T_0]$. Consequently, P has at least one solution $(e, u, w)$ on $[0,T_0]$ for some $T_0 \in (0,T]$. 
\par Next, we will prove the uniqueness of solutions  to P on $[0,T_0]$ for any $T_0 \in (0, T]$. For $0 < T_0 \leq T$ and $i=1,2$ let $(e_i, u_i, w_i)$ be solutions of P on $[0,T_0]$. Namely, by Definition~\ref{th1}, it holds 
\[
\begin{aligned}
&e_i \in W^{1, \infty}(0, T_0), \quad \delta  \leq  e_i \leq  1-\delta  \text{ on } [0, T_0],\\
&u_i, w_i \in W^{1,2}(0,T_0;L^2(0,1)) \cap L^{\infty}(0, T_0; W^{1,2}(0,1)),\\
&u_{ixx}, w_{ixx} \in L^2(Q_l(T_0, e_i)) \cap L^2(Q_a(T_0, e_i)),\\
&u_{ix}(\cdot, e_i\pm) \in L^{\infty}(0,T_0), \quad w_i(\cdot, e_i) \geq \delta_1 \text{ on } [0,T_0].
\end{aligned}
\]
Clearly, $\Gamma(e_i) = e_i$ on $[0,T_0]$ for $i=1,2$. On account of the proof of the estimates, we see that $e_1 = e_2 \text{ on } [0,T_*]$ for some $T_* \in (0,T]$, and the uniqueness of AP($e$) guaranteed by Lemma \ref{lem-w-3} implies that $u_1=u_2$ and $w_1=w_2$ on $(0,T_*) \times (0,1)$. By repeating the argument above finite times, we have proved the uniqueness of solutions to P.
\section{Proof of Theorem \ref{th2}}\label{sec4}
Throughout this section, we suppose that all the assumptions of Theorem \ref{th2} hold.
\begin{lemma}\label{lem500}
 Let $(e,u,w)$ be a solution of P on $[0,T]$. Then, it holds that $u + \theta \leq u_b$ on $(0,T)\times(0,1)$.
\end{lemma}
\begin{proof}
First, \Aref{2} and the assumption $u_b \geq u_0 + \theta_c$ on $[0,1]$ suggests that $u_0(e_0) = 0$ and $u_b \geq \theta_c$. Accordingly, for any $t \in [0, T]$, by multiplying $[u(t) + \theta_c -u_b]^+$ on both sides of  \eqref{EQl} and
integrating it with $x$ on $[0, e(t)]$, $[e(t), 1]$, respectively, we have
\begin{align*}
\int^e_0 c_l u_t  [u + \theta_c - u_b]^+ dx &= \int^e_0 k_l u_{xx}  [u + \theta_c - u_b]^+ dx\\
&=- \int^e_0 k_l u_{x}  ([u + \theta_c - u_b]^+)_x dx\\
&=- \int^e_0 k_l   \left|([u + \theta_c - u_b]^+)_x\right|^2 dx\\
&\leq 0 \text{ a.e. on } [0,T],
\end{align*}
and from the monotonicity of  the boundary condition it follows 
\begin{align*}
&\int^1_e c_a u_t  [u + \theta_c - u_b]^+ dx\\
 &= \int^1_e k_a u_{xx}  [u + \theta_c - u_b]^+ dx\\
&=-(h( u(\cdot,1) + \theta_c - u_b) + \sigma( (u(\cdot,1) + \theta_c)^4 - u_b^4)) [u(\cdot,1) + \theta_c - u_b]^+\\
&\quad \, - \int^1_e k_a u_{x}  ([u + \theta_c - u_b]^+)_x dx\\
&\leq - \int^1_e k_a   \left|([u + \theta_c - u_b]^+)_x\right|^2 dx\\
&\leq 0 \text{ a.e. on } [0,T].
\end{align*}
Thus, we obtain
\begin{align*} 
\int^1_0 u_t  [u + \theta - u_b]^+ dx \leq 0 \text{ a.e. on } [0,T].
\end{align*}
Since $u_b$ is a constant, we see that
\begin{align*} 
\frac{1}{2} \frac{d}{dt} \int^1_0  \left| [u + \theta_c - u_b]^+\right|^2 dx \leq 0 \text{ a.e. on } [0,T].
\end{align*}
The assumption $u_b \geq u_0 + \theta$ on $[0,1]$ implies
\begin{align*} 
 \int^1_0  \left| [u(t) + \theta_c - u_b]^+\right|^2 dx &\leq  \int^1_0  \left| [u_0 + \theta_c - u_b]^+\right|^2 dx\\
&=0 \text{ on } [0,T].
\end{align*}
This shows that Lemma \ref{lem500} holds.
\end{proof}
\begin{lemma}\label{lem5}
 Let $(e,u,w)$ be a solution of P on $[0,T]$. Then there exists a positive constant $\delta_w $ such that $w \geq \delta_w$ on $(0,T)\times(0,1)$.
\end{lemma}
\begin{proof}
By the assumption $w_0 > 0$ on $[0,1]$,  we can choose $\delta_w$ such that $w_0(x) \geq \delta_w$ for $x \in [0,1]$.
For any $t \in [0, T]$ we multiply $[-w(t) + \delta_w]^+$ on both sides of  \eqref{EQa} and
integrate it with $x$ on $[0, e(t)]$, $[e(t), 1]$, respectively. By integration by parts, we have
\begin{align*} 
\int^1_0 w_t  [-w + \delta_w]^+ dx&= -\int^e_0 d_l w_x  ([-w + \delta_w]^+)_x dx - \int^1_e d_a w_x  ([-w + \delta_w]^+)_x dx\\
&\quad \, + d_a w_x(\cdot,1)[-w(\cdot,1) + \delta_w]^+\\
&= \int^e_0 d_l  \left| ([-w + \delta_w]^+)_x\right|^2 dx + \int^1_e d_a \left|  ([-w + \delta_w]^+)_x\right|^2 dx\\
&\quad \, - \{b_1p(u(\cdot,1) + \theta_c) -  b_2 p(u_b)\}[-w(\cdot,1) + \delta_w]^+   \text{ a.e. on } [0,T].
\end{align*}
 Moreover, by the assumption $b_1 \leq b_2$, Lemma \ref{lem500} and the monotonicity of $p$,  it follows that
 \begin{samepage}
  
\begin{align*} 
\frac{1}{2}\frac{d}{dt} \int^1_0 \left| [-w + \delta_w]^+\right|^2 dx &\leq  \{b_1p(u(\cdot,1) + \theta_c) -  b_2 p(u_b)\}[-w(\cdot,1) + \delta_w]^+\\
&\leq 0 \text{ a.e. on } [0,T].
\end{align*}
 \end{samepage}

Thanks to  $w_0(x) \geq \delta_w$ for $x \in [0,1]$, we have
\begin{align*} 
 \int^1_0 \left| [-w + \delta_w]^+\right|^2 dx 
 &\leq \int^1_0 \left| [-w_0 + \delta_w]^+\right|^2 dx\\
  &\leq 0 \quad \ \text{on } [0,T].
\end{align*}
 Thus, Lemma \ref{lem5} has been proved.
\end{proof}
\begin{lemma}\label{lem6}
 Let $[0, T^*)$ be the maximal interval of existence of the solution $(u, w, e)$ to P. Then, there exists a positive constant $C_1$ such that 
\begin{equation}
\int^t_0 \left|u_t\right|^2_{L^2(0,1)}d\tau + \left|u(t)\right|_{W^{1,2}(0,1)}^2 + \int^t_0 \left|e'\right|^3 d\tau \leq C_1 \text{ for } t \in [0,T^*).
\end{equation}

\end{lemma}
\begin{proof}
 Multiply $k_l u_t$, $k_a u_t$ on both sides of the first equation of \eqref{EQl}, the second one and integrate it with $x$ on $[0, e]$, $[e, 1]$, respectively. Then, similarly to the proof of Lemma 4.4 in \cite{AK}, we can obtain
\begin{align*} 
& C_* \int^1_0  \left|u_t\right|^2 dx + \frac{k_l^2}{2} \frac{d}{dt} \int^e_0\left|u_x\right|^2 dx   + \frac{k_a^2}{2} \frac{d}{dt} \int^1_e\left|u_x\right|^2 dx \\
&\quad \, + \frac{k_l^2}{2} u_x(\cdot, e-)^2 e' - \frac{k_a^2}{2} u_x(\cdot, e+)^2 e'\\
&\quad \,  + k_a  (h( u(\cdot,1) + \theta_c - u_b) + \sigma( (u(\cdot,1) + \theta_c)^4 - u_b^4))u_t(\cdot, 1)\\
&\leq 0 \text{ a.e. on } [0,T^*),
\end{align*}
where $C_* = \min \{ c_l k_l, c_a k_a \}$. We note that $u_t(\cdot,1)$ is well-defined by Proposition \ref{prop1}. Easily, we get
\begin{align*} 
 \frac{k_l^2}{2} u_x(\cdot, e-)^2 e' - \frac{k_a^2}{2} u_x(\cdot, e+)^2 e' &= \frac{\left|e'\right|^2}{2}(k_l  u_x(\cdot, e-) + k_a  u_x(\cdot, e+) )l w(\cdot, e)\notag\\
 &=: I_1 \text{ a.e. on } [0,T^*).
\end{align*}
If $e'(t) > 0$ for some $t \in [0,T^*)$, then $k_l  u_x(t, e(t)) > k_a  u_x(t, e(t))$ implies that \[\left|e'(t)\right| = k_l  u_x(t, e(t)-) - k_a  u_x(t, e(t)+).\] Hence, we have
\begin{samepage}
  \begin{align*}
I_1(t) &\geq  \frac{\left|e'(t)\right|^2}{2}(k_l  u_x(t, e(t)-) + k_a  u_x(t, e(t)+) )l w(t, e(t))\\
&\geq  \frac{\left|e'(t)\right|^2}{2}(k_l  u_x(t, e(t)-) - k_a  u_x(t, e(t)+) )l w(t, e(t))\\
&= \frac{\left|e'(t)\right|^3}{2}l^2 w(t, e(t))^2.
\end{align*}
\end{samepage}

Here, we note that $u_x(t,e(t)) \geq 0$ because of $u(t,e(t)) = 0$ and $u(t,x) \geq 0$ for all $x \in [e(t),1]$.
In case $e'(t) \leq 0$  for some $t \in [0,T^*)$, we can get the same inequality. Also, we obtain
\begin{align*}
& (h( u(\cdot,1) + \theta_c - u_b) + \sigma( (u(\cdot,1) + \theta_c)^4 - u_b^4))u_t(\cdot, 1)\\
&=\frac{d}{dt}G + hu'_bu(\cdot, 1) + 4\sigma u_b^3 u'_b u(\cdot, 1) \text{ a.e. on } [0,T^*),
\end{align*}
where
\begin{align*}
G &= \frac{h}{2}(u(\cdot,1) + \theta_c)^2 - h u_b u(\cdot,1) +  \frac{\sigma}{5}(u(\cdot,1) + \theta_c)^5 - \sigma u_b^4 u(\cdot,1) \text{ on } [0,T^*).
\end{align*}
As given in Lemma \ref{prop3-1}, $u(\cdot,1) \geq 0$ on $[0,T^*)$. Accordingly, by applying Young's inequality, we have
\begin{align*}
G &=  \frac{h}{2}(u(\cdot,1) + \theta_c)^2 - h u_b (u(\cdot,1) + \theta_c) + h u_b  \theta_c \\
&\quad \,  +  \frac{\sigma}{5}(u(\cdot,1) + \theta_c)^5 - \sigma u_b^4 (u(\cdot,1) + \theta_c) +  \sigma u_b^4\theta_c\\
&\geq  \frac{h}{2}(u(\cdot,1) + \theta_c)^2 - h u_b (u(\cdot,1) + \theta_c) - \sigma u_b^4 (u(\cdot,1) + \theta_c)\\
&\geq  \frac{h}{4}(u(\cdot,1) + \theta_c)^2 - 2h u_b^2 - \frac{2}{h}\sigma^2 u_b^8\\
&\geq -C_b \text{ a.e. on } [0,T^*),
\end{align*}
where $\displaystyle C_b = 2h u_b^2 - \frac{2}{h}\sigma^2 u_b^8$. 
By putting
\begin{align*} 
E_1 = \frac{k_l^2}{2} \int^e_0\left|u_x\right|^2 dx   + \frac{k_a^2}{2} \int^1_e\left|u_x\right|^2 dx + k_a(G + C_b) \text{ on } [0,T^*),
\end{align*}
we have
\begin{align*} 
& C_* \int^1_0  \left|u_t\right|^2 dx +  \frac{\left|e'\right|^3}{2}l^2 w(\cdot, e)^2+  \frac{d}{dt}E_1 \leq 0 \text{ a.e. on } [0,T^*).
\end{align*}
Therefore, we have
\begin{align*} 
 C_* \int^t_0 \int^1_0  \left|u_t\right|^2 dxd\tau +  \frac{l^2}{2}\int^t_0 \left|e'\right|^3 w(\cdot, e)^2 d\tau + E_1(t) \leq E_1(0) \text{ for } t \in [0,T^*).
\end{align*}
Thus, we conclude that Lemma \ref{lem6} holds.
\end{proof}
\begin{lemma}\label{lem8}
 Let $T^* < \infty$ and $[0, T^*)$ be the maximal interval of existence of the solution $(u, w, e)$ to P. Suppose that neither (b) nor (c) does not hold. Then, there exists $\delta_e > 0$ such that
 $\delta_e \leq e \leq 1- \delta_e$ on $[0, T^*)$.
\end{lemma}
\begin{proof}
From the assumption, there exist $\varepsilon_0 > 0$ and sequences $\{ t_{1n} \}$ and  $\{ t_{2n} \}$ such that $t_{1n} > T_* - \frac{1}{n}$, $t_{2n} > T_* - \frac{1}{n}$, $e(t_{1n}) \geq \varepsilon_0$ and $e(t_{2n}) < 1 - \varepsilon_0$ for any $n \in \mathbb{N}$.
Because of $e \in C([0,T^*))$, we can choose $t_n \in (t_{1n}, t_{2n}) \text{ or } (t_{2n}, t_{1n})$ such that
$\varepsilon_0^* \leq e(t_n) \leq 1- \varepsilon_0^*$, where $\varepsilon_0^* = \min \{ \frac{1}{4}, \varepsilon_0 \}$.
By Lemma \ref{lem6}, for $t_n < t < T^*$, we have
\begin{align}
|e(t) - e(t_n)| &\leq \int^t_{t_n}|e'|dt \leq C_1 (t-t_n)^{\frac{2}{3}}\label{cauchy},
\end{align}
where $C_1$ is a positive constant defined by Lemma \ref{lem6}. 
\par Thus, there exists $\delta_1 > 0$ such that for any $n \in \mathbb{N}$ and $t$ with $0 < t - t_n < \delta_1$, it holds $|e(t) - e(t_n)|  < \frac{\varepsilon_0^*}{2}$.
Hence, we have $e(t) \geq \frac{\varepsilon_0^*}{2}$ for $t_n< t < \min\{T^*, t_n + \delta_1\}$ and any $n \in \mathbb{N}$. Moreover, we can choose $n_0 \in \mathbb{N}$ such that $|t_n - T^*| < \delta_1$ for $n \geq n_0$. Accordingly, for $T^* > t \geq t_{n0}$, $e(t) \geq \frac{\varepsilon_0^*}{2}$ holds.  Since $e(t) > 0$ for $0 \leq t \leq t_{n0}$, there exists $\varepsilon_1 > 0$ such that $e \geq \varepsilon_1$ on $[0,t_{n0}]$. Let $\delta_e^{(1)} = \min \{  \frac{\varepsilon_0^*}{2}, \varepsilon_1 \}$. Then, it follows that
\[(t) \geq \delta_e^{(1)} \text{ for } 0 \leq t < T^*.\] 
Similarly, we obtain 
\[ e \leq 1 - \delta_e \text{ on } \left[0,T^*\right) \text{ for some }  \delta_e^{(2)} > 0.\]
Thus, Lemma \ref{lem8} has been proved.
\end{proof}
\begin{lemma}\label{lem7}
 Under the same assumption as in Lemma \ref{lem8}, there exist
 \[
 \begin{aligned}
 u_* &\in W^{1,2}(0,1),
 &w_* &\in W^{1,2}(0,1),
 &e_* &\in (0,1),
 &\varepsilon' &> 0,
 \end{aligned}
 \]
 such that
 \[
 \begin{aligned}
 \varepsilon' &\leq e_* \leq 1-\varepsilon',&&\\
 u_* &\in W^{2,2}(1-\varepsilon',1),
 &u_* \geq 0 \text{ on } [e_*,1],  & \quad u_* \leq 0 \text{ on } [0,e_*],\\
 -k_a u_{* x}(1) &= g(T^*, u_{*}(1)),
 &w_* \geq \delta_w \text{ on } [0,1],
 &\\
 u(t) &\to u_* \text{ in } L^2(0,1)
 & \text{weakly in } W^{2,2}(1-\varepsilon',1),
 &\\
 w(t) &\to w_* \text{ in } L^2(0,1),
 &e(t) \to e_* \text{ in } \mathbb{R} \text{ as } t \uparrow T^*.&
 \end{aligned}
\]
\end{lemma}
\begin{proof}
First, by Lemma \ref{lem6}, we see that
\begin{align*} 
\left|u(t) - u(t')\right|_{L^2(0,1)} &\leq \int^t_{t'} \left|u_{\tau}\right|_{L^2(0,1)} d\tau \\
&\leq C'(t-t')^{1/2} \text{ for } 0 < t' < t < T^*,
\end{align*}
where $C'$ is a positive constant. This shows that $\{ u(t) \}_{t \uparrow T^*}$ is a Cauchy sequence in $L^2(0,1)$. Similarly, \eqref{cauchy} implies that $\{ e(t) \}_{t \uparrow T^*}$ is also a Cauchy sequence in $\mathbb{R}$. Hence, there exist $e_* \in \mathbb{R}$, $\varepsilon > 0$ and $u_* \in L^2(0,1)$ such that
\begin{align*} 
e(t) \to e_* \text{ in } \mathbb{R} \text{ as } t \uparrow T^*, \quad \varepsilon \leq e_* \leq 1-\varepsilon,
\end{align*}
and
\begin{align*} 
u(t) \to u_* \text{ in } L^2(0,1) \text{ as } t \uparrow T^*.
\end{align*}
Here, we note that $\varepsilon > 0$, since neither (b) nor (c) does not occur.
Moreover, thanks to Lemmas \ref{lem2} and \ref{lem4}, we observe that  $\{ u(t)| 0 \leq t < T^* \}$ is bounded in $W^{1,2}(0,1)$ and $\varepsilon' \leq e \leq 1-\varepsilon'$ on $[0,T^*)$ for some $\varepsilon' > 0$. Hence, by applying Lemma \ref{lem4} and $u_{xx} = \frac{c_a}{k_a}u_t$,  we observe that $\{ u(t)| 0 \leq t < T^* \}$ is also bounded in $W^{2,2}(1-\varepsilon',1)$. Immediately, it holds that
\begin{align*} 
u(t) \to u_* \text{ in } C([0,1]),  \text{ and weakly in } W^{1,2}(0,1) \text{ and } W^{2,2}(1-\varepsilon',1) \text{ as } t \uparrow T^*.
\end{align*}
\par Next, by Lemma \ref{prop3-1}, we see that $u(t) \geq 0$ on $[e(t),1]$ for $t \in [0,T^*)$. Hence, it is easy to obtain $u_* \geq 0$ on $[e_*, 1]$. Similarly, we can prove $u_* \leq 0$ on $[0, e_*]$.
\par From now on, we show that $w_* \geq \delta_w$ on $[0,1]$. First, Lemma \ref{lem-w-1} guarantees that   $\{ w(t) \}_{t \uparrow T^*}$ is a Cauchy sequence in  $L^2(0,1)$. Accordingly, there exists $w_* \in L^2(0,1)$ such that
\begin{align*} 
w(t) \to w_* \text{ in } L^2(0,1) \text{ as } t \uparrow T^*.
\end{align*}
Moreover, the boundedness of  $\{ w(t)| 0 \leq t < T^* \}$ in $W^{1,2}(0,1)$ implies that
\begin{align*} 
w(t) \to w_* \text{ in } C([0,1]) \text{ and weakly in } W^{1,2}(0,1) \text{ as } t \uparrow T^*.
\end{align*}
 Since $w(t) \geq \delta_w$ on $(0,T^*) \times (0,1)$ by Lemma \ref{lem5}, we can show $w_* \geq \delta_w$ on $[0,1]$.
\par Finally, we show $-k_a u_{*  x}(1) = g(T^*, u_{*}(1))$. By Lemma \ref{lem4}, $\{ u_x(t)| 0 \leq t < T^* \}$ is bounded in $W^{1,2}(1-\varepsilon',1)$, $\{ u_x(t_n) \}$ is bounded in $W^{1,2}(1-\varepsilon',1)$. Hence, 
\begin{align*} 
u_x(t) \to u_{*x} \text{ in } C([1-\varepsilon',1]) \text{ as } t \uparrow T^*.
\end{align*}
It is also clear that
\begin{align} 
u(t) \to u_* \text{ in } C([1-\varepsilon',1]) \text{ as } t \uparrow T^*.\label{con8}
\end{align}
By using  \eqref{con8}, we obtain $-k_a u_{*  x}(1) = g(T^*, u_{*}(1))$.
\end{proof}
\begin{proof}[Proof of Theorem \ref{th2}]
Suppose that $T^* < \infty$ and either (b) or (c) does not hold. Then, by Lemma \ref{lem7}, there exists a triplet $(e_*, u_*, w_*)$ satisfying \Aref{1}--\Aref{5} as the initial time $T^*$. Consequently, Theorem \ref{th1} implies existence of a unique solution $(\hat{e}, \hat{u}, \hat{w})$ on $[T^*, \hat{T}]$ for some $\hat{T} > T^*$. Immediately, we can extend the solution $(e,u,w)$ beyond $T^*$ by using $(\hat{e}, \hat{u}, \hat{w})$. This contradicts the definition of $T^* < \infty$. Hence, $T^* = \infty$. This completes the proof of Theorem \ref{th2}.
\end{proof}

\section{Conclusion}
In this paper, we have established existence and uniqueness of a strong solution under high regularity for the initial data and shown the behavior of the free boundary. For future works, we consider the behavior of the free boundary. As a first step, we will prove that (c) in Theorem \ref{th2} does not occur.

\end{document}